\documentclass[12pt,centertags,oneside]{amsart}
\usepackage{amssymb}
\usepackage{amsmath,amstext,amsthm,a4,amssymb,amscd}
\usepackage[mathscr]{eucal}
\usepackage{mathrsfs}
\usepackage{epsf}
\usepackage{hyperref}
\usepackage{charter}
\usepackage{typearea}

\usepackage[english]{babel}
\usepackage[T1]{fontenc}

\usepackage[a4paper,width=16.2cm,top=3cm,bottom=3cm]{geometry}

\numberwithin{equation}{section}
\allowdisplaybreaks[3]

\newcommand{\field}[1]{\mathbb{#1}}
\newcommand{\Z}{\field{Z}}
\newcommand{\R}{\field{R}}
\newcommand{\C}{\field{C}}
\newcommand{\N}{\field{N}}

 \def\cC{\mathscr{C}}
 
\def\cL{\mathscr{L}}

\def\cO{\mathscr{O}}

\def\mH{\mathcal{H}}

\def\mO{\mathcal{O}}

\def\cF{\mathscr{F}}

\def\cL{\mathscr{L}}
\def\cP{\mathscr{P}}
\def\cQ{\mathscr{Q}}
\def\cK{\mathscr{K}}
\def\cT{\mathscr{T}}

 \DeclareMathOperator{\End}{End}
 \DeclareMathOperator{\Ker}{Ker}
 
 \DeclareMathOperator{\rank}{rank}
\DeclareMathOperator{\Id}{Id}

\DeclareMathOperator{\td}{Td} \DeclareMathOperator{\vol}{vol}
\DeclareMathOperator{\ch}{ch}

\newcommand{\norm}[1]{\lVert#1\rVert}

\newcommand{\om}{\omega}

\newtheorem{thm}{Theorem}[section]
\newtheorem{lemma}[thm]{Lemma}
\newtheorem{prop}[thm]{Proposition}
\newtheorem{cor}[thm]{Corollary}
\theoremstyle{definition}
\newtheorem{rem}[thm]{Remark}
\theoremstyle{definition}
\newtheorem{defn}[thm]{Definition}
\newcommand{\be}{\begin{eqnarray}}
\newcommand{\ee}{\end{eqnarray}}
\newcommand{\ov}{\overline}

\newcommand{\var}{\varepsilon}

\newcommand{\comment}[1]{}

\begin{document}


\title[Berezin-Toeplitz quantization for eigenstates
of the Bochner-Laplacian]{Berezin-Toeplitz quantization for eigenstates
of the Bochner-Laplacian on symplectic manifolds}

\author{Louis Ioos}
\address{Institut de Math\'ematiques de Jussieu--Paris Rive Gauche,
UFR de Math{\'e}matiques,
Universit{\'e} Paris Diderot - Paris 7, Case 7012,
75205 Paris Cedex 13,\,France}
\email{louis.ioos@imj-prg.fr}
\thanks{L.\ I.\ was supported by grants from R\'egion Ile-de-France}
\author{Wen Lu}

\address{School of Mathematics and Statistics,
Huazhong University of Science and Technology,
Wuhan, 430074, People's Republic of China}
\email{wlu@hust.edu.cn}
\thanks{W.\ L.\ partially supported by National Nature Science Foundation
in China (No. 11401232)}

\author{Xiaonan Ma}

\address{Institut de Math\'ematiques de Jussieu--Paris Rive Gauche,
UFR de Math{\'e}matiques,
Universit{\'e} Paris Diderot - Paris 7, Case 7012,
75205 Paris Cedex 13, \,France}
\email{xiaonan.ma@imj-prg.fr}
\thanks{X.\ M.\ partially supported by
NNSFC  11528103, ANR-14-CE25-0012-01 and
funded through the Institutional Strategy of
the University of Cologne within the German Excellence Initiative}

\author{George Marinescu}

\address{Universit{\"a}t zu K{\"o}ln,  Mathematisches Institut,
    Weyertal 86-90,   50931 K{\"o}ln, Germany
    \newline
    \mbox{\quad}\,Institute of Mathematics `Simion Stoilow', 
	Romanian Academy,
Bucharest, Romania}
\email{gmarines@math.uni-koeln.de}
\thanks{G.\ M.\ partially supported by DFG funded
project SFB TRR 191}

\dedicatory{To Gennadi Henkin, in memoriam}
\begin{abstract}
We study the Berezin-Toeplitz quantization using as quantum space
the space of eigenstates of the renormalized Bochner Laplacian 
corresponding to
eigenvalues localized near the origin on a 
symplectic manifold. 
We show that this quantization has the correct semiclassical behavior
and construct the corresponding star-product.
%
   \end{abstract}

\maketitle
\tableofcontents
\section{Introduction}\label{s0}
Quantization is a recipe in physics for
passing from a classical system  
to a quantum system by obeying certain natural rules.
By a classical system we understand a classical
phase space (a symplectic manifold) and classical
observables (smooth functions).
The quantum system consists of a quantum space (a Hilbert space of
functions or sections of a bundle)
and quantum observables (bounded linear operators on the quantum space).
The quantum system should reduce to the classical one as the size of 
the objects gets large, 
that is, as the ``Planck constant'',
which, heuristically, corresponds to the magnitude where the quantum 
phenomena
become relevant, tends to zero. This is the so-called classical limit.

The original concept of quantization goes back to Weyl, von Neumann, and
Dirac. In the geometric quantization introduced by Kostant \cite{Kos:70}
and Souriau \cite{Sou:70} the quantum space is the Hilbert space of
square integrable holomorphic sections of a prequantum line bundle
(see also 
\cite{BFFLS,CaGuRa:90,Fedo:96}).
Berezin-Toeplitz quantization is a particularly efficient version of 
the geometric quantization theory.
Toeplitz operators and more generally Toeplitz structures were introduced 
in geometric quantization by Berezin \cite{Berez:74} 
and Boutet de Monvel-Guillemin \cite{BouGu81}. 
Using the analysis of Toeplitz structures 
\cite{BouGu81}, Bordemann-Meinrenken-Schlichenmaier \cite{BMS94} 
and Schlichenmaier \cite{Schlich:00} showed that the Berezin-Toeplitz 
quantization on a compact K\"ahler manifold
satisfies the correspondence principle asymptotically and introduced 
the Berezin-Toeplitz star product (cf.\eqref{e:cp} and \eqref{toe4.4c})
when $E=\C$.


In order to generalize the Berezin-Toeplitz quantization to arbitrary
symplectic manifolds one has to find a substitute for the space of
holomorphic sections
of tensor powers of the prequantum line bundle. A natural candidate is
the kernel of the Dirac operator, since it has similar features to the space
of holomorphic sections in the K\"ahler case, especially the asymptotics
of the kernels of the orthogonal projection on both spaces \cite{Dai04}. 
The Berezin-Toeplitz quantization with quantum space the kernel of the Dirac operator
was carried over by Ma-Marinescu \cite{MM08}.

Another appealing candidate is the space of eigenstates of the renormalized 
Bochner Laplacian \cite{GU,Ma07,Ma08} corresponding to eigenvalues
localized near the origin, cf.\ \eqref{0.6}, \eqref{dp}.
In this paper we construct the Berezin-Toeplitz quantization for these 
spaces and show that it has the correct semiclassical behavior.
The difference between this case and the quantization by the kernel
of the Dirac operator comes from the possible presence of eigenvalues localized
near the origin but different from zero. In this situation the analysis becomes 
more difficult.

 
Let us note also that Charles \cite{Cha14} proposed recently another 
approach to quantization of symplectic manifolds and Hsiao-Marinescu
\cite{HM}
constructed a Berezin-Toeplitz quantization for eigenstates of 
small eigenvalues in the case of complex manifolds.

The readers are referred to the monograph \cite{Ma07} (also \cite{Ma10})
for a comprehensive study of the (generalized) Bergman kernel, 
Berezin-Toeplitz quantization and its applications.

Let us describe the setting and results in detail. 
Let $(X, \omega)$ be a compact symplectic
manifold of real dimension $2n$. Let $(L, h^{L})$ be a
Hermitian line bundle on $X$, and let $\nabla^{L}$ be a 
Hermitian connection on $(L, h^{L})$ with
the curvature $R^{L}=(\nabla^{L})^{2}$. Let $(E, h^{E})$ be a
Hermitian vector bundle with Hermitian connection $\nabla^{E}$.
We will assume
throughout the paper that $L$ is a line bundle
satisfying the pre-quantization condition
\begin{equation}\label{0.1}
\frac{\sqrt{-1}}{2\pi}R^{L}=\omega.
\end{equation}
We choose an almost complex structure $J$
such that $\omega$ is $J$-invariant. The almost complex structure $J$
induces a splitting $TX\otimes_{\R}\C=T^{(1, 0)}X\oplus T^{(0, 1)}X$,
where $T^{(1, 0)}X$ and $T^{(0, 1)}X$
are the eigenbundles of $J$ corresponding to the
eigenvalues $\sqrt{-1}$ and $-\sqrt{-1}$, respectively.

Let $g^{TX}$ be a $J$-invariant Riemannian metric on $TX$.
Let $dv_{X}$ be the Riemannian volume form  of $(X, g^{TX})$. The
$L^{2}$-Hermitian product on the space $\cC^{\infty}(X, L^{p}\otimes E)$ of smooth sections
of $L^{p}\otimes E$ on $X$, with $L^{p}:=L^{\otimes p}$, is given by
\begin{equation}\label{0.2}
\big\langle s_{1}, s_{2}\big\rangle 
=\int_{X}\big\langle s_{1}(x), s_{2}(x)\big\rangle dv_{X}(x).
\end{equation}
Let $\nabla^{TX}$ be the Levi-Civita connection on $(X, g^{TX})$,
and let $\nabla^{L^{p}\otimes E}$
be the connection on $L^{p}\otimes E$ induced by $\nabla^{L}$ and 
$\nabla^{E}$.
Let $\{e_{k}\}$ be a local orthonormal frame of $(TX, g^{TX})$.
The Bochner Laplacian acting on $\cC^{\infty}(X, L^{p}\otimes E)$
is given by
\begin{equation}\label{0.3}
\Delta^{L^{p}\otimes E}
=-\sum_{k}\Big[\big(\nabla_{e_{k}}^{L^{p}\otimes E}\big)^{2}
-\nabla_{\nabla_{e_{k}}^{TX}e_{k}}^{L^{p}\otimes E}\Big].
\end{equation}
Let $\Phi\in \cC^{\infty}(X, \textup{End}(E))$ be Hermitian 
(i.e., self-adjoint with respect to $h^{E}$).
The renormalized Bochner Laplacian is defined by
\begin{align}\label{0.4}
\Delta_{p, \Phi}=\Delta^{L^{p}\otimes E}-\tau p+\Phi,
\end{align}
with
\begin{align}
\tau=\frac{\sqrt{2}}{2}\sum_{k}R^{L}(e_{k}, Je_{k}).
\end{align}
Set
\begin{align}
\mu_{0}=\inf_{u\in T^{(1, 0)}_{x}X,\ x\in X}
R^{L}(u, \overline{u})/|u|^{2}_{g^{TX}}.
\end{align}
By \cite{GU}, \cite[Corollary 1.2]{Ma02}, \cite[Theorem 8.3.1]{Ma07} there exists $C_{L}>0$
independent of $p$ such that
\begin{align}\label{0.5}
\textup{Spec}(\Delta_{p, \Phi})
\subset [-C_{L}, C_{L}]\cup [2\mu_{0} p-C_{L}, +\infty),
\end{align}
where $\textup{Spec}(A)$ denotes the spectrum of the operator $A$.
Since $\Delta_{p, \Phi}$ is an elliptic operator on a compact manifold,
it has discrete spectrum and its eigensections are smooth.
Let
\begin{equation}\label{0.6}
\mH_{p}:=\bigoplus_{\lambda\in[-C_{L}, C_{L}]}
\Ker(\Delta_{p, \Phi}-\lambda)\subset \cC^{\infty}(X, L^{p}\otimes E)
\end{equation}
be the direct sum of eigenspaces of $\Delta_{p, \Phi}$ corresponding 
to the eigenvalues lying in $[-C_{L}, C_{L}]$.

In mathematical physics terms, the operator $\Delta_{p, \Phi}$ is 
a semiclassical Schr\"odinger operator and the space $\mH_p$
is the space of its bound states as $p\to\infty$. 
The space $\mH_p$ proves to be an appropriate replacement for 
the space of holomorphic sections $H^0(X,L^p\otimes E)$ 
from the K\"ahler case. Indeed, if $(X,\omega)$ is K\"ahler,
then $\mH_p=H^0(X,L^p\otimes E)$ for $p$ large enough. Moreover,
for an arbitrary 
compact prequantized symplectic manifold $(X,\omega)$ as above,
the dimension of the space $\mH_p$ is given for $p$ large enough 
as in the K\"ahler case by the Riemann-Roch-Hirzebruch formula,
see \cite[Corollary 1.2]{Ma02}, \cite[Theorem 8.3.1]{Ma07},
\begin{equation}\label{dp}
d_p:=\dim\mH_p=\langle\td(X)\ch(L^p\otimes E),[X]\rangle
\sim p^n(\rank{E})\vol_{\om}(X),\:\:p\gg1.
\end{equation}
Another striking similarity is the
fact that the kernel of the orthogonal projection on $\mH_p$ has an
asymptotic expansion analogous to the Bergman kernel expansion
for K\"ahler manifolds, see \cite{Ma07,Ma08}.
We will use the asymptotic expansion of \cite{Ma07,Ma08} together 
with the approach
of \cite{MM08} to Berezin-Toeplitz quantization
in order to derive the properties of Toeplitz operators modeled on the
projection on $\mH_p$.

Let $P_{\mH_p}$ be the orthogonal projection from
$\cC^{\infty}(X, L^p\otimes E)$ onto $\mH_p$. The kernel
$P_{\mH_p}(x, x')$ of $P_{\mH_p}$ with respect to $dv_{X}(x')$
is called a generalized Bergman kernel \cite{Ma08}.
Note that
$P_{\mH_p}(x, x')\in (L^{p}\otimes E)_{x}
\otimes (L^p\times E)_{x'}^{\ast}$.
For a smooth section $f\in \cC^{\infty}(X, \End(E))$ of the bundle $\End(E)$, 
we define the Berezin-Toeplitz quantization of $f$ by
\begin{equation}\label{0.6a}
T_{f, p}=P_{\mH_{p}}fP_{\mH_{p}}\in \End(L^{2}(X, L^p\otimes E)),
\end{equation}
where we denote for simplicity by $f$ the endomorphism of 
$L^{2}(X, L^p\otimes E)$ induced by $f$, namely,
$s\mapsto fs$, with $(fs)(x)=f(x)(s(x))$, for $s\in L^{2}(X, L^p\otimes E)$
and $x\in X$.

\begin{defn}\label{d3.1}
A Toeplitz operator is a sequence $\{T_{p}\}=\{T_{p}\}_{p\in \N}$ 
of linear operators
\begin{align}\label{0.16}
T_{p}: L^{2}(X, L^{p}\otimes E)\longrightarrow L^{2}(X, L^{p}\otimes E)
\end{align}
with the following properties:
\\
(i) For any $p\in \N$, we have
\begin{align}\label{0.17}
T_{p}=P_{\mH_{p}}T_{p}P_{\mH_{p}};
\end{align}
(ii) There exist a sequence $g_{l}\in \cC^{\infty}(X, \End(E))$ such that
for all $k\geqslant 0$ there exists $C_{k}>0$ with
\begin{align}\label{0.18}
\Big\|T_{p}-P_{\mH_{p}}\Big(\sum^{k}_{l=0}p^{-l}g_{l}\Big)P_{\mH_{p}}
\Big\|   \leqslant  C_{k} p^{-k-1},
\end{align}
where $\|\cdot\|$ denotes the operator norm on the space of 
the bounded operators.

If each $T_{p}$ is self-adjoint, then $\{T_{p}\}_{p\in \N}$ is called 
self-adjoint.
\end{defn}
We express (\ref{0.18}) symbolically by
\begin{align}\label{0.19}
T_{p}=\sum^{k}_{l=0}p^{-l}T_{g_{l}, p}+ \mO(p^{-k-1}).
\end{align}
If (\ref{0.18}) holds for any $k\in \N$, then we write
\begin{align}\label{0.20}
T_{p}=\sum^{\infty}_{l=0}p^{-l}T_{g_{l}, p}+ \mO(p^{-\infty}).
\end{align}

The main result of this paper is as follows.
\begin{thm}\label{t0}
Let $(X, J, \omega)$ be a compact symplectic manifold, 
$(L, h^{L}, \nabla^{L})$,
$(E, h^{E}, \nabla^{E})$ be Hermitian vector bundles as above, 
and $g^{TX}$ be
an $J$-compatible Riemannian metric on $TX$.
Let $f, g\in \cC^{\infty}(X, \End(E))$. Then the product of the Toeplitz
operators $T_{f, p}$ and $T_{g, p}$ is a Toeplitz operator, more precisely,
it admits the asymptotic expansion in the sense of \eqref{0.20}:
\begin{align}\label{toe4.2}
T_{f, p}T_{g, p}=\sum^{\infty}_{r=0}p^{-r}T_{C_{r}(f, g), p}
+\mO(p^{-\infty}),
\end{align}
where $C_{r}$ are bidifferential operators,
$C_{r}(f, g)\in \cC^{\infty}(X,\End(E))$
and $C_{0}(f, g)=fg$.

If $f, g\in \cC^{\infty}(X)$, then we have
\begin{align}\label{e:c1d}
C_{1}(f, g)-C_{1}(g, f)=\sqrt{-1}\{f, g\}\Id_{E},
\end{align}
where $\{\cdot, \cdot\}$ is the Poisson bracket on $(X, 2\pi \omega)$, 
and therefore the correspondence principle holds asymptotically,
\begin{equation}\label{e:cp}
[T_{f, g}, T_{g, p}]=\frac{\sqrt{-1}}{p}T_{\{f, g\}, p}+O(p^{-2})
,\:\: p\to\infty.
\end{equation}
\end{thm}
\begin{cor}
Let $f,g\in\cC^\infty(X,\End(E))$. Set
\begin{equation}\label{toe4.4c}
f*g:=\sum_{k=0}^\infty C_k(f,g) \hbar^{k}
\in\cC^\infty(X,\End(E))[[\hbar]].
\end{equation}
where $C_{r}(f,g)$ are determined by \eqref{toe4.2}.
Then \eqref{toe4.4c} defines an associative star-product on 
$\cC^\infty(X,\End(E))$.
\end{cor}
\begin{thm}\label{t2}
For any $f\in\cC^\infty(X,\End(E))$ the operator norm of $T_{f,\,p}$
satisfies
\begin{equation}\label{toe4.17}
\lim_{p\to\infty}\norm{T_{f,\,p}}={\norm f}_\infty
:=\sup_{0\neq u\in E_x, x\in X} |f(x)(u)|_{h^{E}}/ |u|_{h^{E}}\,.
\end{equation}
\end{thm}
In the special case when the Riemannian metric $g^{TX}$ is 
associated with $\omega$
we can even calculate $C_1(f,g)$, not only the
difference $C_{1}(f, g)-C_{1}(g, f)$ from \eqref{e:c1d}.
To state the result, let
\begin{equation}
\begin{split}
& (\nabla^{E})^{1, 0}:\cC^{\infty}(X, \End(E))\rightarrow 
\cC^{\infty}(X, T^{*(1,0)}X\otimes\End(E)),\\
&(\nabla^{E})^{0, 1}:\cC^{\infty}(X, \End(E))\rightarrow 
\cC^{\infty}(X, T^{*(0,1)}X\otimes\End(E)).
\end{split}
\end{equation}
be the $(1, 0)$-component and $(0, 1)$-component 
respectively of the connection $\nabla^{E}$, and 
let $\langle\cdot,\cdot\rangle$  denote the pairing induced by 
$g^{TX}$ on $T^*X\otimes \End(E)$ with values in $\End(E)$.

Following an argument of \cite{Ioos}, we get the last result of this paper.

\begin{thm}\label{t3}
If $g^{TX}(\cdot, \cdot)=\omega(\cdot, J\cdot)$,
then  for any $f, g\in\cC^{\infty}(X,\End(E))$, the coefficient 
$C_1(f,g)\in\cC^{\infty}(X,\End(E))$ defined in \eqref{toe4.2} 
is given by 
\begin{equation}\label{c1fg}
C_{1}(f, g)=\big\langle (\nabla^{E})^{1, 0}f, 
(\nabla^{E})^{0, 1}g \big\rangle.
\end{equation}
\end{thm}

Note that this formula is clearly compatible with the formula \eqref{e:c1d} 
for the Poisson bracket in the case $f, g\in\cC^{\infty}(X)$. 
Note also that \eqref{c1fg} is a direct generalization of the formula
\cite[(0.20)]{Ma12} for K\"ahler manifolds.

We completed this paper a while ago.
Recently Kordyukov informed us
about his preprint \cite{Kor:17} in which 
 the Berezin-Toeplitz quantization by eigenstates of the Bochner-Laplacian
is reconsidered.

The organization of the paper is as follows. In Section \ref{s1}, we
recall the asymptotic expansion of the generalized Bergman kernel
obtained in \cite{LMM16}. In Section \ref{s2}, we obtain the asymptotic
expansion of the kernel of a Toeplitz operator.
In Section \ref{s3}, we show that the asymptotic expansion is also
a sufficient condition for a family to be Toeplitz.
In Section \ref{s4}, we conclude that the set of Toeplitz operators
forms an algebra.
In Section \ref{s6} we prove Theorem \ref{t3}.

\section{The asymptotic expansion of the generalized 
Bergman kernel}\label{s1}

Let $a^{X}$ be the injectivity radius of $(X, g^{TX})$.
Let $d(x, y)$ denote the Riemannian distance from $x$ to $y$ on 
$(X, g^{TX})$.
By \cite[Proposition 8.3.5]{Ma07} (cf.\ also \cite[(1.11)]{LMM16}), 
we have the following far off-diagonal behavior of the generalized Bergman
kernel.
\begin{prop}\label{t1}
For any $b>0$ and any $k, l\in \N$ and $0<\theta<1$, there exists
$C_{b, k, l, \theta}>0$ such that
\begin{align}\label{0.7}
\Big|P_{\mH_{p}}(x, x')\Big|_{\cC^{k}(X\times X)}\leqslant
C_{b, k, l, \theta}\, p^{-l},
\ \
\textup{for}\ d(x, x')>b p^{-\frac{\theta}{2}},
\end{align}
here the $\cC^{k}$-norm is induced by 
$\nabla^{L}, \nabla^{E}, h^{L}, h^{E}$ and $g^{TX}$.
\end{prop}

Let $\var\in (0, a^X/4)$ be fixed. We denote by $B^{X}(x, \var)$ and
$B^{T_{x}X}(0, \var)$ the open balls
in $X$ and $T_{x}X$ with center $x$ and radius $\var$, respectively.
We identify $B^{T_{x}X}(0, \var)$ with $B^{X}(x, \var)$
by using the exponential map of $(X, g^{TX})$.

Let $x_0\in X$.
 For $Z\in B^{T_{x_0}X}(0,\var)$ we identify
 $(L_Z, h^L_Z)$, $(E_Z, h^E_Z)$ and $(L^p\otimes E)_Z$
 to $(L_{x_0},h^L_{x_0})$, $(E_{x_0},h^E_{x_0})$ and $(L^p\otimes E)_{x_0}$
 by parallel transport with respect to the connections
 $\nabla ^L$, $\nabla ^E$ and $\nabla^{L^p\otimes E}$ along the curve
 $\gamma_Z :[0,1]\ni u \to \exp^X_{x_0} (uZ)$.
This is the basic trivialization we use in this paper.

Using this trivialization we identify $f\in \cC^\infty(X,\End(E))$ to a family
$\{f_{x_0}\}_{x_0\in X}$ where $f_{x_0}$ is the function $f$ in
normal coordinates near $x_0$, i.\,e.,
$f_{x_0}:B^{T_{x_0}X}(0,\var)\to\End(E_{x_0})$,
$f_{x_0}(Z)=f\circ\exp^X_{x_0}(Z)$\,.
In general, for functions expressed in normal coordinates centered at $x_0\in X$
we will add a subscript $x_0$ to indicate the base point $x_0$.

Similarly,
$P_{\mH_p}(x,x')$ induces in terms of the basic trivialization a smooth section
\[
(Z,Z')\longmapsto P_{\mH_p,\,x_0}(Z,Z')
\]
of $\pi ^* \End(E)$ over $\{(Z,Z')\in TX\times_{X} TX:|Z|,|Z'|<\var\}$,
which depends smoothly on $x_0$. Here 
$\pi: TX\times_{X} TX\rightarrow X$ is the natural projection from
the fibered product $TX\times_{X} TX$ on $X$ and
we identify a section
$S\in \cC^\infty \big(TX\times_{X}TX,\pi ^* \End (E)\big)$
 with the family $(S_x)_{x\in X}$, where
 $S_x=S|_{\pi^{-1}(x)}$. 

Let $dv_{TX}$ be the Riemannian volume form on 
$(T_{x_{0}}X, g^{T_{x_{0}}X})$.
Let $\kappa_{x_{0}}(Z)$ be the smooth positive function defined 
by the equation
\begin{align}\label{0.8}
dv_{X}(Z)=\kappa_{x_{0}}(Z)dv_{TX}(Z),\ \ \kappa_{x_{0}}(0)=1,
\end{align}
where the subscript $x_{0}$ of $\kappa_{x_{0}}(Z)$ indicates
the base point $x_{0}\in X$.

We identify the $2$-form $R^{L}$ with the Hermitian matrix
$\dot{R}^{L}\in\End(T^{(1, 0)}X)$ such that
for any $W, Y\in T^{(1, 0)}X$,
\begin{align}
R^{L}(W, Y)=\big\langle \dot{R}^{L}W, Y\big\rangle.
\end{align}
Choose $\{\omega_{j}\}^{n}_{j=1}$ an orthonormal basis of 
$T_{x_{0}}^{(1, 0)}X$ such that
\begin{align}
\dot{R}^{L}(x_{0})=\textup{diag}(a_{1}, \ldots, a_{n}) \in
\End(T_{x_{0}}^{(1, 0)}X).
\end{align}
We fix an orthonormal basis of $T_{x_{0}}X$ given by
$e_{2j-1}=\frac{1}{\sqrt{2}}(\omega_{j}+\ov{\omega}_{j})$ and
$e_{2j}=\frac{\sqrt{-1}}{\sqrt{2}}(\omega_{j}-\ov{\omega}_{j})$.
Then $z_{j}=Z_{2j-1}+\sqrt{-1}Z_{2j}$ is a complex coordinate of
$Z\in \R^{2n}\simeq (T_{x_{0}}X, J)$.

By \cite[Theorem 2.1]{LMM16} and \cite[Theorem 1.18]{MM08},
we obtain the following version of the off diagonal expansion of
the generalized Bergman kernel.

\begin{thm}\label{t22}
For any $x_{0}\in X$ and $r\in \N$, there exist
polynomials $J_{r, x_{0}}(Z, Z')\in \End(E_{x_{0}})$ in $Z, Z'$ with
the same parity as $r$ and with $\textup{deg} J_{r, x_{0}}\leqslant 3r$ 
such that by setting
\begin{align}\label{0.9}
\cF_{r, x_{0}}(Z, Z')=J_{r, x_{0}}(Z, Z')\cP(Z, Z'),\
\ J_{0, x_{0}}(Z, Z')=\Id_{E_{x_{0}}},
\end{align}
with
\begin{align}\label{p}
\cP(Z, Z')=\exp\Big[-\frac{1}{4}\sum^{n}_{j=1}a_{j}
\big(|z_{j}|^{2} +|z_{j}'|^{2}-2z_{j}\ov{z}'_{j}\big)\Big],
\end{align}
the following statement holds:
for any $b>0$ and $k_{0}, m, m'\in \N$,
there exists $C_{b, k_{0}, m, m'}>0$ such that
for $|\alpha|+|\alpha'|\leqslant m'$ and any
$|Z|, |Z'|<b p^{-\frac{1}{2}+\theta}$
with
\begin{align}\label{0.10}
\theta=\frac{1}{2\big(2n+8+2k_{0}+3m'+2m\big)}\,,
\end{align}
we have
\begin{align}\label{0.11}
\begin{split}
& \bigg|
\frac{\partial^{|\alpha|+|\alpha'|}}
{\partial Z^{\alpha}\partial Z'^{\alpha'}}
\Big(p^{-n}P_{\mH_p, x_{0}}(Z, Z')
-\sum^{k}_{r=0}\cF_{r, x_{0}}(\sqrt{p}Z, \sqrt{p}Z')
   \kappa_{x_{0}}^{-1/2}(Z)
\kappa_{x_{0}}^{-1/2}(Z')p^{-\frac{r}{2}}\Big)
\bigg|_{\cC^{m}(X)}\\
& \quad\leqslant
C_{b, k_{0}, m, m'}\,p^{-\frac{k_{0}}{2}-1},
\end{split}\end{align}
where $k=k_{0}+m'+2$ and $\cC^{m}(X)$ is the $\cC^{m}$ norm for
the parameter $x_{0}\in X$.
\end{thm}

In particular when $m'=0$, the following statement holds:
for any $b>0$ and $k, m\in \N$, there exists $C>0$ such that
for any $|Z|, |Z'|<b p^{-\frac{1}{2}+\theta_{2}}$
with
\begin{align}\label{0.12}
\theta_{2}=\frac{1}{4\big(n+k+m+2\big)}\,,
\end{align}
we have
\begin{align}\begin{split}\label{0.13}
& \bigg|p^{-n}P_{\mH_p, x_{0}}(Z, Z')
-\sum^{k}_{r=0}\cF_{r, x_{0}}(\sqrt{p}Z, \sqrt{p}Z')
\kappa_{x_{0}}^{-1/2}(Z)
\kappa_{x_{0}}^{-1/2}(Z')p^{-\frac{r}{2}}
\bigg|_{\cC^{m}(X)}\leqslant 
Cp^{-k/2}.
\end{split}\end{align}

\noindent
Note that the more expansion term in \eqref{0.13}, the smaller of 
the expansion domain for the variables $Z$ and $Z'$. 
This serves as the main ingredient for the generalized Bergman kernel case.

By \cite[Lemma\,2.2]{MM08}, for any polynomials $F, G\in \C[Z, Z']$
there exist $\cK[F, G]\in \C[Z,Z']$ such that
\begin{align}\label{k1}
\big((F\cP)\circ(G\cP)\big)(Z, Z')=\cK[F, G](Z, Z')\cP(Z, Z').
\end{align}

\section{Asymptotic expansion of Toeplitz operators}\label{s2}

For $f\in \cC^{\infty}(X, \End(E))$ we define the Toeplitz operator
$T_{f, p}$ on $L^{2}(X, L^{p}\otimes E)$ by \eqref{0.6a}.
%
The Schwartz kernel of $T_{f, p}$ is given by
\begin{align}\label{0.22}
T_{f, p}(x, x')=\int_{X}P_{\mH_{p}}(x, x'')f(x'')
P_{\mH_{p}}(x'', x')dv_{X}(x'').
\end{align}
Note that if $f\in \cC^{\infty}(X, \End(E))$ is self-adjoint,
i.e., $f(x)=f(x)^{\ast}$
for all $x\in X$, then the operator $T_{f, p}$ is self-adjoint.

We examine now the asymptotic expansion of the kernel
of Toeplitz operators $T_{f, p}$.
Outside the diagonal of $X\times X$, we have the following analogue 
of \cite[Lemma\,4.2]{MM08}.

\begin{lemma}\label{t3.2}
Let $\theta\in (0, 1)$ and $f\in \cC^{\infty}(X, \End(E))$ fixed.
For any $b>0$ and $k, l\in \N$, there exists $C_{b, k, l}>0$ such that
\begin{align}\label{0.23}
\Big|T_{f, p}(x, x')\Big|_{\cC^{k}(X\times X)}
\leqslant C_{b, k, l} p^{-l},
\end{align}
for all $p\geqslant 1$ and all $(x, x')\in X\times X$
with $d(x, x')>b p^{-\theta}$,
where the $\cC^{k}$-norm is induced by 
$\nabla^{L}, \nabla^{E}, h^{L}, h^{E}$ and $g^{TX}$.
\end{lemma}

\begin{proof}
From Proposition \ref{t1} and (\ref{0.13}), we know that
for any $k\in \N$ there exist $C_{k}>0$ and $M_{k}>0$ such that
for all $(x, x')\in X\times X$,
\begin{align}\label{0.24}
\Big|P_{\mH_{p}}(x, x')\Big|_{\cC^{k}(X\times X)}
\leqslant C_{k}p^{M_{k}}.
\end{align}
We split the integral in \eqref{0.22} in a sum of two integrals as follows:
\begin{align}\label{0.24a}
T_{f, p}(x, x')=\Big(\int_{B^{X}(x, \frac{b}{2}p^{-\theta})}
+\int_{X\setminus B^{X}(x, \frac{b}{2}p^{-\theta})}\Big)
P_{\mH_{p}}(x, x'')f(x'')P_{\mH_{p}}(x'', x')dv_{X}(x'').
\end{align}
Assume that $d(x, x')> bp^{-\theta}$. Then
\begin{align}\label{0.24b}
\begin{split}
&d(x'', x')> \frac{b}{2}p^{-\theta}\quad \text{for $x''\in B^{X}(x, \frac{b}{2}p^{-\theta})$},\\
&d(x, x'')\geqslant \frac{b}{2}p^{-\theta}\quad 
\text{for $x''\in X\big\backslash B^{X}(x, \frac{b}{2}p^{-\theta})$}.
\end{split}\end{align}
Now from \eqref{0.7} and \eqref{0.24}-\eqref{0.24b},
we get \eqref{0.23}.
The proof of Lemma \ref{t3.2} is complete.
\end{proof}

We concentrate next on a neighborhood of the diagonal of $X\times X$
in order to
obtain the asymptotic expansion of the kernel $T_{f, p}(x, x')$.

Let $\{\Xi_{p}\}_{p\in \N}$ be a sequence of linear operators
$\Xi_{p}: L^{2}(X, L^{p}\otimes E)\rightarrow L^{2}(X, L^{p}\otimes E)$
with smooth kernel $\Xi_{p}(x, y)$ with respect to $dv_{X}(y)$.
Recall that $\pi: TX\times_{X} TX\rightarrow X$ is the natural projection. 
Under our trivialization,
$\Xi_{p}(x, y)$ induces a smooth section $\Xi_{p, x_{0}}(Z, Z')$ of
$\pi^{\ast}\big(\End(E)\big)$ over $TX\times_{X} TX$ with
$Z, Z'\in T_{x_{0}}X$, $|Z|, |Z'|<a^{X}$. Recall also that
$\cP_{x_{0}}=\cP$ was defined by (\ref{p}).

Consider the following condition for $\{\Xi_{p}\}_{p\in \N}$.
\medskip

\noindent
{\bf Condition A.}
 There exists
a family $\{Q_{r, x_{0}}\}_{r\in \N, x_{0}\in X}$ such that
\\[2pt]
(a) $Q_{r, x_{0}}\in \End(E_{x_{0}})[Z, Z']$;
\\[2pt]
(b) $\{Q_{r, x_{0}}\}_{r\in \N, x_{0}\in X}$ is smooth with
respect to the parameter $x_{0}\in X$ and there exist
$b_{1}, b_{0}\in\N$ such that $\deg Q_{r}\leqslant b_{1}r+b_{0}$;
\\[2pt]
(c) for any $k, m\in \N$, there exists $\theta_{k, m}\in (0, 1/2)$
such that for any $b>0$,
there exist $C_{b, k, m}>0$
such that for every
$x_{0}\in X$, $Z, Z'\in T_{x_{0}}X$ with
$|Z|, |Z'|<b p^{-\frac{1}{2}+\theta_{k, m}}$
and $p\in \N^{\ast}$, the following estimate holds:
\begin{align}\begin{split}\label{0.25}
 & \bigg|p^{-n}\Xi_{p, x_{0}}(Z, Z')\kappa^{1/2}_{x_{0}}(Z)
 \kappa_{x_{0}}^{1/2}(Z')
-\sum^{k}_{r=0}(Q_{r, x_{0}}\cP_{x_{0}})(\sqrt{p}Z, \sqrt{p}Z')p^{-r/2}
\bigg|_{\cC^{m}(X)}
 \leqslant
C_{b, k, m}p^{-k/2}.
\end{split}\end{align}
\medskip

\noindent
{\bf Notation A.} For any $k, m\in \N$ we write
the equation (\ref{0.25})
 for $|Z|, |Z'|<b p^{-\frac{1}{2}+\theta_{k, m}}$ as
\begin{align}\label{0.26}
p^{-n}\Xi_{p, x_{0}}(Z, Z')\cong
\sum^{k}_{r=0}(Q_{r, x_{0}}\cP_{x_{0}})(\sqrt{p}Z, \sqrt{p}Z')p^{-r/2}
+\cO_{m}(p^{-k/2}).
\end{align}

 \begin{rem} \label{t3.3}
 By Theorem \ref{t22}, (\ref{0.12}) and (\ref{0.13}), we have
 \begin{align}\label{0.28}
 p^{-n}P_{\mH_{p}, x_{0}}(Z, Z')\cong
\sum^{k}_{r=0}(J_{r, x_{0}}\cP_{x_{0}})(\sqrt{p}Z, \sqrt{p}Z')p^{-r/2}
+\cO_{m}(p^{-k/2}),
\end{align}
in the sense of Notation A with
\begin{align}
\theta_{k, m}=\frac{1}{4(n+k+m+2)}\ \
\textup{for}\ k, m\in \N,
\end{align}
where $J_{r, x_{0}}(Z, Z')\in \End(E_{x_{0}})$
are the polynomials in $Z, Z'$ defined in \eqref{0.9}. Note that
$J_{r, x_{0}}(Z, Z')$ has the same parity as $r$ and
$\deg J_{r, x_{0}}\leqslant 3r$, $J_{0, x_{0}}=\Id_{E_{x_{0}}}$.
\end{rem}

The following result is about the near diagonal
asymptotic expansion of the kernel $T_{f, p}(x, x')$.
It is a version of \cite[Lemma\,4.6]{MM08} in our situation.

\begin{lemma}\label{t3.4}
Let $f\in \cC^{\infty}(X, \End(E))$. There exists a family
$\{Q_{r, x_{0}}(f)\}_{r\in \N, x_{0}\in X}$ such that
\\[2pt]
\textup{(a)} $Q_{r, x_{0}}(f)\in \End(E_{x_{0}})[Z, Z']$
are polynomials with the same parity as $r$;
\\[2pt]
\textup{(b)} $\{Q_{r, x_{0}}(f)\}_{r\in \N, x_{0}\in X}$ is
smooth with respect to $x_{0}\in X$, and $\deg Q_{r, x_{0}}\leqslant 3r$;
\\[2pt]
\textup{(c)} for any $k_{0}, m\in \N$, we have
\begin{align}\label{0.29}
 p^{-n}T_{f, p, x_{0}}(Z, Z')\cong
\sum^{k_{0}}_{r=0}\big(Q_{r, x_{0}}(f)\cP_{x_{0}}\big)
(\sqrt{p}Z, \sqrt{p}Z')p^{-r/2}
+\cO_{m}(p^{-k_{0}/2}),
\end{align}
in the sense of Notation A with
\begin{align}
\theta_{k_{0}, m}=\frac{1}{4(n+k+m+2)}\ \textup{for some}\ 
k\geqslant k_{0},
\end{align} and $Q_{r, x_{0}}(f)$ are expressed by
\begin{align}\label{30}
Q_{r, x_{0}}(f)=\sum_{r_{1}+r_{2}+|\alpha|=r}\cK\Big[J_{r_{1}, x_{0}},
\frac{\partial^{\alpha}f_{x_{0}}}{\partial Z^{\alpha}}(0)
\frac{Z^{\alpha}}{\alpha!}J_{r_{2}, x_{0}}\Big].
\end{align}
Especially,
\begin{align}\label{3.17}
Q_{0, x_{0}}(f)=f(x_{0})\Id_{E_{x_{0}}}.
\end{align}
\end{lemma}

\begin{proof}
For $k_{0}, m\in \N$ fixed, let $k\geqslant k_{0}$ to be determined.
Set
\begin{align}\label{3.18}
\theta_{2}=\frac{1}{4(n+k+m+2)},\ \ \theta_{1}=1-2\theta_{2}.
\end{align}
By (\ref{0.13}), we have for any
$|Z|, |Z'|< 2b p^{-\frac{1}{2}+\theta_{2}}=2b p^{-\theta_{1}/2}$,
\begin{align}\begin{split}\label{0.18a}
& \Big|p^{-n}P_{\mH_{p}, x_{0}}(Z, Z')-
\sum^{k}_{r=0}\cF_{r, x_{0}}(\sqrt{p}Z, \sqrt{p}Z')
\kappa^{-1/2}_{x_{0}}(Z)\kappa^{-1/2}_{x_{0}}(Z')p^{-r/2}
\Big|_{\cC^{m}(X)}\\
&\hspace{3cm}\leqslant 
C_{k, l}p^{-k/2}.
\end{split}\end{align}


\noindent
For $|Z|, |Z'|< \frac{b}{2} p^{-\frac{1}{2}+\theta_{2}}
=\frac{b}{2}p^{-\theta_{1}/2}$, we get
from (\ref{0.22}) that
\begin{align}\label{3.17a}
T_{f, p, x_{0}}(Z, Z')=\int_{X}P_{\mH_{p}, x_{0}}(Z, y)f(y)
P_{\mH_{p}, x_{0}}(y, Z')dv_{X}(y).
\end{align}
We split the integral into integrals over $B^{X}(x, b p^{-\theta_{1}/2})$ 
and $X\setminus B^{X}(x, b p^{-\theta_{1}/2})$. 
We have 
\begin{equation}\label{3.22}
d(y, \exp_{x_{0}}Z)\geqslant d(y, x_{0})-|Z|
> \frac{b}{2} p^{-\theta_{1}/2}\,\:\:\text{on
$X\setminus B^{X}(x, b p^{-\theta_{1}/2})$},
\end{equation}
since on this set $d(y, x_{0})> b p^{-\theta_{1}/2}$ holds.
By Proposition \ref{t1} for $\theta_{1}$ in (\ref{3.18}),
(\ref{0.24}) and (\ref{3.22}) we have
for $|Z|, |Z'|< \frac{b}{2} p^{-\theta_{1}/2}$,
\begin{equation}\label{3.23}
T_{f, p, x_{0}}(Z, Z')=\!\!
\int\limits_{|Z''|<b p^{-\frac{\theta_{1}}2}}\!\!
P_{\mH_{p}, x_{0}}(Z, Z'')f_{x_{0}}(Z'')P_{\mH_{p}, x_{0}}(Z'', Z')
\kappa_{x_{0}}(Z'')dv_{TX}(Z'')+\cO_{m}(p^{-\infty})
\end{equation}
Then
\begin{align}\begin{split}\label{0.38}
& p^{-n}T_{f, p, x_{0}}(Z, Z')\kappa_{x_{0}}^{1/2}(Z)
\kappa_{x_{0}}^{1/2}(Z')
 \\  &= 
p^{-n}\!\!\int\limits_{|Z''|<bp^{-\frac{\theta_{1}}2}}
\!\!\!P_{\mH_{p}, x_{0}}(Z, Z'')\kappa_{x_{0}}^{\frac12}(Z)
\kappa_{x_{0}}^{\frac12}(Z'')
 f_{x_{0}}(Z'')
 P_{\mH_{p}, x_{0}}(Z'', Z')\kappa_{x_{0}}^{\frac12}(Z'')
 \kappa_{x_{0}}^{\frac12}(Z')dv_{TX}(Z'')\\
&\hspace{3cm}+\cO_{m}(p^{-\infty}).
 \end{split}\end{align}
 We consider the Taylor expansion of $f_{x_{0}}$:
 \begin{align}\label{0.39}
 f_{x_{0}}(Z)=&
 \sum_{|\alpha|\leqslant k} \frac{\partial^{\alpha}f_{x_{0}}}
 {\partial Z^{\alpha}}(0)
 \frac{Z^{\alpha}}{\alpha!}+ O(|Z|^{k+1})
  \\ = &
 \sum_{|\alpha|\leqslant k} p^{-|\alpha|/2}
 \frac{\partial^{\alpha}f_{x_{0}}}{\partial Z^{\alpha}}(0)
 \frac{(\sqrt{p}Z)^{\alpha}}{\alpha!}+ p^{-\frac{k+1}{2}}
 O(|\sqrt{p}Z|^{k+1}).
 \nonumber \end{align}
 Combining the asymptotic expansion (\ref{0.18a}) and (\ref{0.39}),
 to obtain the asymptotic expansion of (\ref{0.38}),
 we need to consider $I_{r_{1}, |\alpha|, r_{2}}(T_{x_{0}}X)(Z, Z')$
 defined by
 \begin{equation}\label{3.21}
 \begin{split}
 &p^{-n+\frac{r_{1}+|\alpha|+r_{2}}{2}} 
 I_{r_{1}, |\alpha|, r_{2}}(T_{x_{0}}X)(Z, Z'):= \\
 &\int_{T_{x_{0}}X}(J_{r_{1}, x_{0}}\cP_{x_{0}})(\sqrt{p}Z, \sqrt{p}Z'')
\frac{\partial^{\alpha}f_{x_{0}}}{\partial Z^{\alpha}}(0)
 \frac{(\sqrt{p}Z'')^{\alpha}}{\alpha!}
 (J_{r_{2}, x_{0}}\cP_{x_{0}})(\sqrt{p}Z'', \sqrt{p}Z')dv_{TX}(Z'').
 \end{split}
 \end{equation}
 Clearly, we can define
 $I_{r_{1}, |\alpha|, r_{2}}\big(B^{T_{x_{0}}X}(0, a)\big)(Z, Z')$ and
 $I_{r_{1}, |\alpha|, r_{2}}\big(T_{x_{0}}X
 \backslash B^{T_{x_{0}}X}(0, a)\big)
 (Z, Z')$ for $a>0$ in the same manner.
 Then by (\ref{0.38}),
 \begin{align}\label{3.23a}
 & p^{-n}T_{f, p, x_{0}}(Z, Z')\kappa^{1/2}(Z)\kappa^{-1/2}(Z')
\nonumber \\
= &\sum_{r_{1}, |\alpha|, r_{2}\leqslant k}
 I_{r_{1}, |\alpha|, r_{2}}
 \big(B^{T_{x_{0}}X}(0, bp^{-\theta_{1}/2})\big)(Z, Z')
  \\ &
 +I_{1}(Z, Z')+I_{2}(Z, Z')+I_{3}(Z, Z')
 +\cO_{m}(p^{-\infty}),
 \nonumber \end{align}
 with
 \begin{align}\label{3.23b}
 & I_{1}(Z, Z')
 \\ =&
 \int_{|Z''|<bp^{-\theta_{1}/2}}
 \Big[p^{-n}P_{\mH_{p}}(Z, Z'')\kappa^{1/2}(Z)\kappa^{1/2}(Z'')
 -\sum_{r\leqslant k}(J_{r, x_{0}}
 \cP_{x_{0}})(\sqrt{p}Z, \sqrt{p}Z'')p^{-r/2}\Big]
 \nonumber \\
 & \times f_{x_{0}}(Z'')P_{\mH_{p}}(Z'', Z')
 \kappa^{1/2}(Z'')\kappa^{1/2}(Z')dv_{TX}(Z''),
 \nonumber
 \end{align}
 and
 \begin{align}
 & I_{2}(Z, Z')
 \\ =&
 \int_{|Z''|<bp^{-\theta_{1}/2}}\sum_{r_{1}\leqslant k}(J_{r_{1}, x_{0}}
 \cP_{x_{0}})(\sqrt{p}Z, \sqrt{p}Z'')p^{-r_{1}/2}
 \nonumber \\ & \times
 \Big[f_{x_{0}}(Z'')-\sum_{|\alpha|\leqslant k}
 \frac{\partial^{\alpha}f_{x_{0}}}
 {\partial Z^{\alpha}}(0)\frac{(\sqrt{p}Z'')^{\alpha}}{\alpha!}
 p^{-|\alpha|/2}\Big]
 P_{\mH_{p}}(Z'', Z')\kappa^{1/2}(Z'')\kappa^{1/2}(Z')dv_{TX}(Z''),
\nonumber \\
& I_{3}(Z, Z')
\nonumber
\\ &=
 p^{n}\int_{|Z''|<bp^{-\theta_{1}/2}}\sum_{r_{1}
 \leqslant k}(J_{r_{1}, x_{0}}\cP_{x_{0}})(\sqrt{p}Z, \sqrt{p}Z'')
 p^{-r_{1}/2}
 \sum_{|\alpha|\leqslant k}\frac{\partial^{\alpha}f_{x_{0}}}
 {\partial Z^{\alpha}}(0)\frac{(\sqrt{p}Z'')^{\alpha}}{\alpha!}
 p^{-|\alpha|/2}
 \nonumber \\ & \times
 \Big[p^{-n}P_{\mH_{p}}(Z'', Z')\kappa^{1/2}(Z'')\kappa^{1/2}(Z')
 -\sum_{r_{2}\leqslant k}(J_{r_{2}}\cP_{x_{0}})(\sqrt{p}Z'', 
 \sqrt{p}Z')p^{-r_{2}/2}\Big]dv_{TX}(Z'').
\nonumber
\end{align}
We claim that for $k$ large,
\begin{align}\label{3.24}
\big|I_{j}(Z, Z')\big|_{\cC^{m}(X)}\leqslant C p^{-k_{0}/2}\ \
\textup{for}\ j=1, 2, 3.
\end{align}
In fact, by (\ref{0.24}), there exists $C_{0}>0$ and $M_{0}>0$
such that for all $(x, x')\in X\times X$,
\begin{align}\label{3.24a}
\big|P_{\mH_{p}}(x, x')\big|_{\cC^{0}(X\times X)}
\leqslant C_{0}p^{M_{0}}.
\end{align}
Combining  (\ref{0.18a}), (\ref{3.23b}) and (\ref{3.24a}) yields
\begin{align}\label{3.24b}
|I_{1}(Z, Z')|_{\cC^{m}(X)}\leqslant Cp^{-\frac{k}{2}+M_{0}}.
\end{align}
 By (\ref{0.39}), (\ref{3.24a}) and the fact that $\deg J_{r}\leqslant 3r$,
 \begin{align}\begin{split}\label{3.25}
 \big|I_{2}(Z, Z')\big|_{\cC^{m}(X)}
 &\leqslant   C (1+\sqrt{p}|Z|)^{3k}\cdot p^{-\frac{k+1}{2}}
 \cdot p^{M_{0}}
 \\ &\leqslant 
 C p^{-\frac{k+1}{2}+3k\theta_{2}+M_{0}}.
 \end{split}\end{align}
 By (\ref{0.18a}) and the fact that $\deg J_{r}\leqslant 3r$,
 we have for $|Z|, |Z'|<\frac{b}{2} p^{-\theta_{1}/2}$,
 \begin{align}\label{3.25a}
 \big|I_{3}(Z, Z')\big|_{\cC^{m}(X)}
 \leqslant C (1+\sqrt{p}|Z|)^{3k} p^{-\frac{k}{2}}.
 \end{align}
 From (\ref{3.24b})--(\ref{3.25a}), choose $k>k_{0}$ big enough
 such that
 \begin{align}\label{0.25c}
 k+1-6k\theta_{2}-2M_{0}=k\Big(1-\frac{3}{4(n+k+m+2)}\Big)
 -2M_{0}+1>k_{0}.
 \end{align}
 Then the claim (\ref{3.24}) holds.
 By (\ref{3.23a}) and (\ref{3.24}),
 \begin{align}\begin{split}\label{3.26b}
 & \Big|p^{-n}T_{f, p, x_{0}}(Z, Z')\kappa^{1/2}(Z)\kappa^{-1/2}(Z')
 \\ & \ \ \ \ \
 -\sum_{r_{1}, \alpha, r_{2}\leqslant k}
 I_{r_{1}, |\alpha|, r_{2}}(B^{T_{x_{0}}X}(0, bp^{-\theta_{1}/2}))(Z, Z')
 \Big|_{\cC^{m}(X)}
 \leqslant C p^{-k_{0}/2}.
 \end{split}\end{align}
 Note by (\ref{p}),
\begin{align}\label{3.26a}
 \Big|\cP(Z, Z')\Big|=e^{-\frac{1}{4}\sum_{j}a_{j}|z_{j}-z'_{j}|^{2}}
 \leqslant e^{-\frac{1}{4} \mu_{0}|Z-Z'|^{2}}.
 \end{align}
 Then
 \begin{align}\label{42a}
 \Big|\cP(\sqrt{p}Z, \sqrt{p}Z')\Big|=
 e^{-\frac{p}{4}\sum_{j}a_{j}|z_{j}-z'_{j}|^{2}}
 \leqslant e^{-\frac{p}{4} \mu_{0}|Z-Z'|^{2}}.
 \end{align}
By (\ref{42a}) and the fact that $\deg J_{r}\leqslant 3r$, we obtain
\begin{align}\begin{split}\label{3.34}
& \Big|I_{r_{1}, |\alpha|, r_{2}}\big(T_{x_{0}}
X\backslash B^{T_{x_{0}}X}(0, bp^{-\theta_{1}/2})\big)(Z, Z')\Big|_{\cC^{m}(X)}
 \\ &\leqslant 
C p^{n}\int_{|Z''|>b p^{-\theta_{1}/2}}
(1+\sqrt{p}|Z|+\sqrt{p}|Z''|)^{3r_{1}}
(1+\sqrt{p}|Z|+\sqrt{p}|Z''|)^{3r_{2}}
 \\ & \ \ \ \ \ \
\times (\sqrt{p}|Z''|)^{|\alpha|}\
\exp\big[-\frac{\mu_{0}}{2}\sqrt{p}|Z-Z''|
-\frac{\mu_{0}}{2}\sqrt{p}|Z''-Z'|\big]dv_{TX}(Z'').
\end{split} \end{align}
\noindent
Note that for any $|Z|, |Z'|< \frac{b}{2} p^{-\theta_{1}/2}$ and
$|Z''|> b p^{-\theta_{1}/2}$, we have
\begin{align}\label{44}
|Z|< |Z''|,\ \ |Z'|<|Z''|,\ \
|Z-Z''|\geqslant \frac{1}{2}|Z''|, \ \
|Z'-Z''|\geqslant \frac{1}{2}|Z''|.
\end{align}
Substituting (\ref{44}) into (\ref{3.34}) yields
for any $|Z|, |Z'|< \frac{b}{2} p^{-\theta_{1}/2}$,
\begin{align}\label{3.35}
& \Big|I_{r_{1}, |\alpha|, r_{2}}\big(T_{x_{0}}
X\backslash B^{T_{x_{0}}X}
(0, b p^{-\theta_{1}/2})\big)(Z, Z')\Big|_{\cC^{m}(X)}
 \\
&\leqslant  
C p^{n} \int_{|Z''|>b p^{-\theta_{1}/2}}
\big(1+\sqrt{p}|Z''|\big)^{3(r_{1}+r_{2})}
(\sqrt{p}|Z''|)^{|\alpha|}
e^{-\frac{\mu_{0}}{2}\sqrt{p}|Z''|} dv_{TX}(Z'')
\nonumber \\ 
&\leqslant 
C p^{n}\exp\left(-\frac{b}{4}\mu_{0} p^{\theta_{2}}\right)
 \int_{|Z''|>b p^{-\theta_{1}/2}}
\big(1+\sqrt{p}|Z''|\big)^{3(r_{1}+r_{2})+|\alpha|}
e^{-\frac{\mu_{0}}{4}\sqrt{p}|Z''|} dv_{TX}(Z'')
\nonumber \\ 
&\leqslant 
C \exp\left(-\frac{b}{4}\mu_{0} p^{\theta_{2}}\right)
\int_{|Z''|>b p^{\theta_{2}}}
\big(1+|Z''|\big)^{3(r_{1}+r_{2})+|\alpha|}
e^{-\frac{\mu_{0}}{4}|Z''|} dv_{TX}(Z'')
\nonumber \\ 
&\leqslant 
C \exp\left(-\frac{b}{4}\mu_{0} p^{\theta_{2}}\right).
\nonumber
\end{align}
Combining (\ref{3.26b}) and (\ref{3.35}), we obtain
\begin{align}\label{3.30}
\Big|p^{-n}T_{f, p, x_{0}}(Z, Z')\kappa^{1/2}(Z)\kappa^{1/2}(Z')
-\sum_{r_{1}, |\alpha|, r_{2}\leqslant k}
I_{r_{1}, |\alpha|, r_{2}}(T_{x_{0}}X)(Z, Z')\Big|_{\cC^{m}(X)}
\leqslant C p^{-\frac{k_{0}}{2}}.
\end{align}
Clearly,
\begin{align}\begin{split}\label{3.31a}
& \sum_{r_{1}, |\alpha|, r_{2}\leqslant k}
I_{r_{1}, |\alpha|, r_{2}}(T_{x_{0}}X)(Z, Z')
\\&
=
\Big(\sum_{r_{1}+|\alpha|+r_{2}\leqslant k_{0}}
+ \sum^{3k}_{r_{1}+|\alpha|+r_{2}=k_{0}+1} \Big)
I_{r_{1}, |\alpha|, r_{2}}(T_{x_{0}}X)(Z, Z').
\end{split}\end{align}
By (\ref{k1}) and (\ref{3.21}),
 \begin{align}\begin{split}\label{3.27}
 & I_{r_{1}, |\alpha|, r_{2}}(T_{x_{0}}X)(Z, Z')
 \\ &= 
 p^{-(r_{1}+|\alpha|+r_{2})/2}
 \Big(\cK\big[J_{r_{1}, x_{0}}, 
 \frac{\partial^{\alpha}f_{x_{0}}}{\partial Z^{\alpha}}(0)
 \frac{Z^{\alpha}}{\alpha!} J_{r_{2}, x_{0}}\big]\cP\Big)
 (\sqrt{p}Z, \sqrt{p}Z').
 \end{split}\end{align}
In view of \eqref{3.30}-\eqref{3.27}, to finish the proof of Lemma \ref{t3.4},
 it suffices to prove that the $\cC^{m}$ norm with respect to the
 parameter $x_{0}\in X$ of the term
 \begin{align}
 \sum^{3k}_{r_{1}+|\alpha|+r_{2}=k_{0}+1}
 I_{r_{1}, |\alpha|, r_{2}}(T_{x_{0}}X)(Z, Z'),
 \ \textup{for}\ |Z|, |Z'|< \frac{b}{2} p^{-\theta_{1}/2},
 \end{align}
 is controlled by $ C p^{-k_{0}/2}$ for large $k$.

Estimating $I_{r_{1}, |\alpha|, r_{2}}(T_{x_{0}}X)(Z, Z')$
 for $|Z|, |Z'|<\frac{b}{2} p^{-\theta_{1}/2}$,
 using \eqref{42a}, \eqref{3.27} and
 the fact that $\deg J_{r}\leqslant 3r$, we obtain
\begin{align}\begin{split}\label{43}
& \big|I_{r_{1}, |\alpha|, r_{2}}(T_{x_{0}}X)(Z, Z')\big|_{\cC^{m}(X)}
\\ &\leqslant 
C p^{-s/2}
\big(1+\sqrt{p}|Z|+\sqrt{p}|Z'|\big)^{3s}
\exp\big[-\frac{p}{4} \mu_{0}|Z-Z'|^{2} \big]
 \\  &\leqslant  
C p^{-\frac{s}{2}} p^{\frac{1-\theta_{1}}{2}\cdot 3s}
=Cp^{-s(1-6\theta_{2})/2},
\end{split}\end{align}
with $s=r_{1}+|\alpha|+r_{2}$. If $s>k_{0}$, then
\begin{align}
s(1-6\theta_{2})\geqslant (k_{0}+1)\left(1-\frac{6}{4(n+k+m+2)}\right).
\end{align}
Choose $k$ big enough such that
\begin{align}\label{3.32a}
(k_{0}+1)\left(1-\frac{6}{4(n+k+m+2)}\right)>k_{0}.
\end{align}
Then
\begin{align}\label{3.31}
\Big|I_{r_{1}, |\alpha|, r_{2}}(T_{x_{0}}X)(Z, Z')\Big|_{\cC^{m}(X)}
\leqslant C p^{-k_{0}/2}\ \
\textup{for}\ r_{1}+|\alpha|+r_{2}=s>k_{0}.
\end{align}

To sum up, we have proved the following statament: for fixed $k_{0}$,
choose $k>k_{0}$ such that (\ref{0.25c}) and (\ref{3.32a}) hold. Set
\begin{align}
\theta_{2}=\frac{1}{4(n+k+m+2)}, \ \ \ \theta_{1}=1-2\theta_{2}.
\end{align}
Then for any $|Z|, |Z'|<\frac{b}{2}p^{-\theta_{1}/2}$, we have
\begin{align}\label{3.46}
& \Big|p^{-n}T_{f, p, x_{0}}(Z, Z')\kappa^{1/2}_{x_{0}}(Z)
\kappa^{1/2}_{x_{0}}(Z')
-\sum^{k_{0}}_{r=0}\big(Q_{r, x_{0}}(f)\cP_{x_{0}}\big)
(\sqrt{p}Z, \sqrt{p}Z')p^{-r/2}\Big|_{\cC^{m}(X)}
\\
&\leqslant 
 C p^{-k_{0}/2}, \nonumber
\end{align}
where $Q_{r, x_{0}}(f)$ is given by (\ref{30}).
This completes the proof of Lemma \ref{t3.4}.
\end{proof}

\begin{rem}\label{t5}
Let $\Xi_{p}$ be a sequence of operators
satisfying Condition A and assume that $\Xi_{p}=P_{\mH_{p}}\Xi_{p}P_{\mH_{p}}$
for all $p\in\N$.
Applying the proof of Lemma \ref{t3.4}, by splitting
integrals and studying different integration regions,
we deduce by Theorem \ref{t22} and \eqref{0.25}:
%

For any $k, m, m'\in \N$, there exist $\theta_{k, m, m'}\in (0, 1/2)$
such that for any $b>0$, there exist $C>0$ such that
for every $x_{0}\in X$, $Z, Z'\in T_{x_{0}}X$ with
$|Z|, |Z'|< b p^{-\frac{1}{2}+\theta_{k, m, m'}}$ and
$p\in \N^{\ast}$, $|\alpha|+|\alpha'|\leqslant m'$, we have
\begin{align}\begin{split}\label{3.48}
& \bigg|\frac{\partial^{|\alpha|+|\alpha'|}}
{\partial Z^{\alpha}\partial Z'^{\alpha'}}
\Big(p^{-n}\Xi_{p, x_{0}}(Z, Z')\kappa^{1/2}_{x_{0}}(Z)
\kappa_{x_{0}}^{1/2}(Z')
 -\sum^{k}_{r=0}(Q_{r, x_{0}}\cP_{x_{0}})(\sqrt{p}Z, \sqrt{p}Z')p^{-r/2}
\Big)\bigg|_{\cC^{m}(X)}
\\ & \leqslant
Cp^{-(k-m')/2}.
\end{split}\end{align}
\end{rem}

\noindent
In fact, by $\Xi_{p}=P_{\mH_{p}}\Xi_{p}P_{\mH_{p}}$, 
for $|Z|, |Z'|< b p^{-\frac{1}{2}+\theta_{k, m, m'}}$, we have 
the analogue of (\ref{3.17a}):
\begin{align}\label{3.49}
\Xi_{p, x_{0}}(Z, Z')=
\int_{X}P_{\mH_{p, x_{0}}}(Z, y)\Xi_{p}(y, Z')
P_{\mH_{p}}(y, Z')dv_{X}(y).
\end{align}
Then the estimate (\ref{3.48}) follows 
from Theorem \ref{t22}, (\ref{0.25}), (\ref{0.18a}) and (\ref{3.49}) 
in the same manner as (\ref{3.46}) follows
from (\ref{0.18a}), (\ref{3.17a}) and (\ref{0.39}).

\section{A criterion for Toeplitz operators}\label{s3}

In this section we prove a useful criterion which ensures that a given family
of bounded linear operators is a Toeplitz operator.

\begin{thm}\label{t4.1}
Let $\{T_{p}: L^{2}(X, L^{p}\otimes E)
\rightarrow L^{2}(X, L^{p}\otimes E)\}$ be a family
of bounded linear operators which satisfies the following three conditions:
\\[2pt]
\textup{(i)}\ For any $p\in \N$, $P_{\mH_{p}}T_{p}P_{\mH_{p}}=T_{p}$.
\\[2pt]
\textup{(ii)}\ For any $b>0$, $l\in \N$ and $0<\theta<1$,
there exists $C_{b, l, \theta}>0$ such that for all $p\geqslant 1$ and all
$(x, x')\in X\times X$ with $d(x, x')>b p^{-\theta/2}$,
\begin{align}\label{4.1}
\Big|T_{p}(x, x')\Big|\leqslant C_{b, l, \theta} p^{-l}.
\end{align}
\\[2pt]
\textup{(iii)} There exists a family of polynomials
$\big\{\cQ_{r, x_{0}}\in \End(E_{x_{0}})[Z, Z']\big\}_{x_{0}\in X}$
such that
\\[1pt]
\textup{(a)} each $\cQ_{r, x_{0}}$ has the same parity as $r$
and there exist $b_{1}, b_{0}\in \N$ such
that $\deg \cQ_{r}\leqslant b_{1}r+b_{0}$,
\\[1pt]
\textup{(b)} the family is smooth in $x_{0}\in X$ and
\\[1pt]
\textup{(c)} for any $k_{0}, m\in \N$, there exists 
$\theta_{k_{0}, m}\in (0, 1/2)$
such that for any $b>0$, $p\in \N^{\ast}$, $x_{0}\in X$ and
every $Z, Z'\in T_{x_{0}}X$ with
$|Z|, |Z'|< b p^{-\frac{1}{2}+\theta_{k_{0}, m}}$, we have
\begin{align}\label{4.3}
p^{-n}T_{p, x_{0}}(Z, Z')\cong
\sum^{k_{0}}_{r=0}\big(\cQ_{r, x_{0}}\cP_{x_{0}}\big)(\sqrt{p}Z, \sqrt{p}Z')p^{-r/2}
+\cO_{m}(p^{-k_{0}/2}),
\end{align}
in the sense of Notation A for $k_{0}, m, \theta_{k_{0}, m}$.
\\
Then $\{T_{p}\}$ is a Toeplitz operator.
\end{thm}

\begin{rem}
By Lemmas \ref{t3.2} and \ref{t3.4}, and by (\ref{0.17}), (\ref{0.18})
and the Sobolev inequality (cf.\ \cite[(4.14)]{Dai04}), it follows that every
Toeplitz operator in
the sense of Definition \ref{d3.1} verifies the Conditions (i), (ii) and (iii)
of Theorem \ref{t4.1}.
\end{rem}

We start the proof of Theorem \ref{t4.1}.
Let $T^{\ast}_{p}$ be the adjoint of $T_{p}$. By writing
\begin{align}\label{4.4}
T_{p}=\frac{1}{2}(T_{p}+T^{\ast}_{p})
+\sqrt{-1}\frac{1}{2\sqrt{-1}}(T_{p}-T^{\ast}_{p}),
\end{align}
we may and will assume from now on that $T_{p}$ is self-adjoint.

We will define inductively the sequence $(g_{l})_{l\geqslant 0}$, 
$g_{l}\in \cC^{\infty}(X, \End(E))$ such
that
\begin{align}\label{4.5}
T_{p}=\sum^{q}_{l=0}P_{\mH_{p}}g_{l}p^{-l}P_{\mH_{p}}
+\mO(p^{-q-1}) \ \
\textup{for all}\ q\geqslant 0.
\end{align}
Moreover, we can make these $g_{l}$'s to be self-adjoint.

Let us start with the case $q=0$ of (\ref{4.5}). For an arbitrary 
but fixed $x_{0}\in X$, we set
\begin{align}\label{4.6}
g_{0}(x_{0})=\cQ_{0, x_{0}}(0, 0)\in \End(E_{x_{0}}).
\end{align}
We will show that
\begin{align}\label{4.7}
p^{-n}(T_{p}-T_{g_{0}, p})_{x_{0}}(Z, Z')\cong \cO_{m}(p^{-1}),
\end{align}
which implies the case $q=0$ of (\ref{4.5}), namely,
\begin{align}\label{4.8}
T_{p}=P_{\mH_{p}}g_{0}P_{\mH_{p}}+\mO(p^{-1}).
\end{align}
The proof of (\ref{4.7})--(\ref{4.8}) will be done in Proposition \ref{t4.3} 
and Proposition \ref{t4.9}.

\begin{prop}\label{t4.3}
In the conditions of Theorem \textup{\ref{t4.1}}, we have
\begin{align}
\cQ_{0, x_{0}}(Z, Z')=\cQ_{0, x_{0}}(0, 0)\in \End(E_{x_{0}})
\end{align}
for all $x_{0}\in X$ and all $Z, Z'\in T_{x_{0}}X$.
\end{prop}

\begin{proof}
The proof is divided in the series of Lemmas \ref{t4.4}--\ref{t4.8}.
Our first observation is as follows.

\begin{lemma}\label{t4.4}
$\cQ_{0, x_{0}}\in \End(E_{x_{0}})[Z, Z']$ and $\cQ_{0, x_{0}}$ is a
polynomial in $z, \ov{z}'$.
\end{lemma}
\begin{proof}
By (\ref{4.3}), for $k_{0}=2$ there exists $\theta_{3}\in (0, 1/2)$
such that for any $b>0$ and every $Z, Z'\in T_{x_{0}}X$
with $|Z|, |Z'|<b p^{-\frac{1}{2}+\theta_{3}}$, we have
\begin{align}\label{4.11}
p^{-n}T_{p, x_{0}}(Z, Z')
\cong
\sum^{2}_{r=0}(\cQ_{r, x_{0}}\cP_{x_{0}})
(\sqrt{p}Z, \sqrt{p}Z')p^{-r/2}+\cO_{m}(p^{-1}).
\end{align}
By (\ref{0.28}),
\begin{align}\label{4.12}
p^{-n}P_{\mH_{p}, x_{0}}(Z, Z')
\cong
\sum^{2}_{r=0}(J_{r, x_{0}}\cP_{x_{0}})
(\sqrt{p}Z, \sqrt{p}Z')p^{-r/2}+\cO_{m}(p^{-1}),
\end{align}
in the sense of Notation A with $\theta_{4}=1/4(n+m+4)$.
Combining (\ref{4.11}) and (\ref{4.12}),
modelled the way we get (\ref{3.46}) from (\ref{0.18a}) and (\ref{0.39}),
we obtain that there exists $k\geqslant 2$ large enough and
\begin{align}
\theta_{m}=\frac{1}{4(n+k+m+2)}\in (0, 1/2)
\end{align}
such that for every $Z, Z'\in T_{x_{0}}X$
with $|Z|, |Z'|<b p^{-\frac{1}{2}+\theta_{m}}$,
\begin{align}\begin{split}\label{4.13}
& p^{-n}(P_{\mH_{p}}T_{p}P_{\mH_{p}})_{x_{0}}(Z, Z')
\\ \cong &
\sum^{2}_{r=0}\sum_{r_{1}+r_{2}+r_{3}=r}
\Big[(J_{r_{1}, x_{0}}\cP_{x_{0}})\circ (\cQ_{r_{2}, x_{0}}\cP_{x_{0}})
\circ (J_{r_{3}, x_{0}}\cP_{x_{0}})\Big]
(\sqrt{p}Z, \sqrt{p}Z')p^{-r/2}
 \\ & +\cO_{m}(p^{-1}).
\end{split} \end{align}
Since $P_{\mH_{p}}T_{p}P_{\mH_{p}}=T_{p}$,
we deduce from (\ref{4.11}) and (\ref{4.13}) that
\begin{align}\label{4.15}
\cQ_{0, x_{0}}\cP_{x_{0}}=\cP_{x_{0}}
\circ (\cQ_{0, x_{0}}\cP_{x_{0}})\circ \cP_{x_{0}}.
\end{align}
By \cite[(2.8)]{MM08} and (\ref{4.15}), we obtain
\begin{align}
\cQ_{0, x_{0}}\in \End(E_{x_{0}})[z, \ov{z}'].
\end{align}
The proof of Lemma \ref{t4.4} is complete.
\end{proof}

For simplicity we denote in the rest of the proof 
$F_{x}=\cQ_{0, x}\in \End(E_{x})$.
Let $F_{x}=\sum_{i\geqslant 0} F^{(i)}_{x}$ be the decomposition 
of $F_{x}$ in homogeneous polynomials
$F^{(i)}_{x}$ of degree $i$. We will show $F^{(i)}_{x}$
vanish identically for $i>0$,
that is
\begin{align}\label{4.17}
F^{(i)}_{x}(z, \ov{z}')=0 \ \ \textup{for all}\ i>0\
\textup{and}\ z, \ov{z}'\in \C.
\end{align}
The first step is to prove
\begin{align}\label{4.18}
F^{(i)}_{x}(0, \ov{z}')=0 \ \ \textup{for all}\ i>0\ 
\textup{and}\ z'\in \C.
\end{align}
Since $T_{p}$ is self-adjoint, then we have
\begin{align}\label{4.19}
F^{(i)}_{x}(z, \ov{z}')=\big(F^{(i)}_{x}(z', \ov{z})\big)^{\ast}.
\end{align}
Consider $0<\theta_{k_{0}, m}<1$ as in
hypothesis (iii)(c) of Theorem \ref{t4.1}.
For $Z'\in \R^{2n}\simeq T_{x}X$
with $|Z'|<\var p^{-\frac{1}{2}+\theta_{k_{0}, m}}$
and $y=\exp^{X}_{x}(Z')$. Set
\begin{align}\label{4.20}
F^{(i)}(x, y)=F^{(i)}_{x}(0, \ov{z}')\in \End(E_{x}),
\nonumber \\
\tilde{F}^{(i)}(x, y)=\big(F^{(i)}(y, x)\big)^{\ast}\in\End(E_{y}).
\end{align}
$F^{i}$ and $\tilde{F}^{(i)}$ define smooth sections on a neighborhood 
of the diagonal of $X\times X$. Clearly, the $\tilde{F}(x, y)$'s need not be
polynomials in $z$ and $\ov{z}'$.

Since we wish to define global operators induced by these kernels, we use
a cut-off function in the neighborhood of the diagonal. Pick a
smooth function $\eta\in \cC^{\infty}(\R)$ such that
\begin{align}
\eta(u)=1 \ \textup{for}\ |u|\leqslant \var/2 \ \ \ \textup{and}\ \
\eta(u)=0 \ \textup{for}\ |u|\geqslant \var.
\end{align}

We denote by $F^{(i)}P_{\mH_{p}}$ and $P_{\mH_{p}}\tilde{F}^{(i)}$
the operators defined by the kernels
\begin{align}\label{4.22}
\eta(d(x, y))F^{(i)}(x, y)P_{\mH_{p}}(x, y)\ \ \textup{and}\ \
\eta(d(x, y))P_{\mH_{p}}(x, y)\tilde{F}^{(i)}(x, y)
\end{align}
with respect to $dv_{X}(y)$. Set
\begin{align}\label{4.23}
\cT_{p}=T_{p}-\sum_{i\leqslant \textup{deg}
F_{x}}(F^{(i)}P_{\mH_{p}})p^{i/2}.
\end{align}
The operators $\cT_{p}$ extends naturally to bounded operators 
on $L^{2}(X, L^{p}\otimes E)$.

From (\ref{4.3}) and (\ref{4.23}) we deduce that for any 
$k_{0}, m\in \N$,
there exist $\theta_{k_{0}, m}\in (0, 1/2)$
such that for any $|Z'|<\var p^{-\frac{1}{2}+\theta_{k_{0}, m}}$, 
we have the following expansion in the normal coordinates
around $x_{0}\in X$ (which has to be understood in the sense 
of (\ref{0.26})):
\begin{align}\label{4.25}
p^{-n}\cT_{p, x_{0}}(0, Z')\cong
\sum_{r=1}^{k_{0}}(R_{r, x_{0}}\cP_{x_{0}})(0, \sqrt{p}Z')p^{-r/2}
+\cO_{m}(p^{-k_{0}/2}),
\end{align}
for some polynomials $R_{r, x_{0}}$ of the same parity as $r$. 
For simplicity let us define similarly
to (\ref{4.20}) the kernel
\begin{align}\label{4.26}
R_{r, p}(x, y)=p^{n}(R_{r, x}\cP_{x})(0, \sqrt{p}Z')
\kappa^{-1/2}_{x}(Z')\eta(d(x, y)),
\end{align}
where $y=\exp_{x}^{X}(Z')$, and denote by $R_{r, p}$
the operator defined by this kernel.

\begin{lemma}\label{t4.5}
There exists $C>0$ such that for every $p\geqslant 1$ and
$s\in L^{2}(X, L^{p}\otimes E)$, we have
\begin{align}\label{4.27}
\big\|\cT_{p}s\big\|_{L^{2}}\leqslant Cp^{-1/2}\big\|s\big\|_{L^{2}},
\nonumber \\
\big\|\cT^{\ast}_{p}s\big\|_{L^{2}}\leqslant Cp^{-1/2}
\big\|s\big\|_{L^{2}}.
\end{align}
\end{lemma}

\begin{proof}
In order to use (\ref{4.25}) we write
\begin{align}
\big\|\cT_{p}s\big\|_{L^{2}}
\leqslant
\Big\|\big(\cT_{p}-\sum^{k_{0}}_{r=1}p^{-r/2}
R_{r, p}\big)s\Big\|_{L^{2}}
+\Big\|\sum^{k_{0}}_{r=1}p^{-r/2}R_{r, p}s\Big\|_{L^{2}}.
\end{align}
By the Cauchy-Schwarz inequality we have
\begin{align}\begin{split}\label{4.30}
&\Big\|\big(\cT_{p}-\sum^{k_{0}}_{r=1}p^{-r/2}R_{r, p}\big)s
\Big\|_{L^{2}}^{2}
\\ &
\leqslant \int_{X} \Big(\int_{X}\Big|\big(\cT_{p}
-\sum^{k_{0}}_{r=1}p^{-r/2}R_{r, p}\big)(x, y)\Big|dv_{X}(y)\Big)
 \\ & \ \ \ \ \ \ \ \ \ \
\times
\Big(\int_{X}\Big|\big(\cT_{p}-\sum^{k_{0}}_{r=1}
p^{-r/2}R_{r, p}\big)(x, y)\Big|\big|s(y)\big|^{2}dv_{X}(y)\Big)
dv_{X}(x).
\end{split}
\end{align}
By (\ref{0.7}), (\ref{4.1}), (\ref{4.22}), (\ref{4.23}) and (\ref{4.26}),  
we obtain uniformly in $x\in X$,
\begin{align}\begin{split}\label{4.31}
& \int_{X}\Big|\big(\cT_{p}-\sum^{k_{0}}_{r=1}
p^{-r/2}R_{r, p}\big)(x, y)\Big|\big|s(y)\big|^{2}dv_{X}(y)
 \\
\leqslant &
 \int_{B^{X}(x, \frac{\var}{2} p^{-\frac{1}{2}+\theta_{k_{0}, m}})}\Big|
 \big(\cT_{p}-\sum^{k_{0}}_{r=1}p^{-r/2}R_{r, p}\big)(x, y)
 \Big|\big|s(y)\big|^{2}dv_{X}(y)
 \\ &
\ \ \ \ \ \ \ \ \ \ \ \ \ \ \ \ \  +
 O(p^{-\infty})
 \int_{X\backslash B^{X}(x, \frac{\var}{2} p^{-\frac{1}{2}
 +\theta_{k_{0}, m}})}
 |s(y)|^{2}dv_{X}(y),
\end{split}\end{align}
where $\theta_{k, m}$ is given by (\ref{4.25}).
By (\ref{0.25}) and (\ref{4.25}) we obtain
\begin{align}\begin{split}\label{4.32}
 & \int_{B^{X}(x, \frac{\var}{2} p^{-\frac{1}{2}+\theta_{k_{0}, m}})}
 \Big|\big(\cT_{p}-\sum^{k_{0}}_{r=1}p^{-r/2}R_{r, p}\big)(x, y)
 \Big|\big|s(y)\big|^{2}dv_{X}(y)
  \\
 &= O(p^{-1}) \int_{B^{X}(x, \frac{\var}{2}
 p^{-\frac{1}{2}+\theta_{k_{0}, m}})}\big|s(y)\big|^{2}dv_{X}(y).
\end{split}\end{align}
In the same vein (by splitting the integral region as above) we obtain
\begin{align}\label{4.33}
\int_{X}\Big|\big(\cT_{p}-\sum^{k_{0}}_{r=1}p^{-r/2}
R_{r, p}\big)(x, y)\Big|dv_{X}(y)
=O(p^{-1})+O(p^{-\infty}).
\end{align}
Combining (\ref{4.30})--(\ref{4.33}) yields
\begin{align}\label{83}
\Big\|\Big(\cT_{p}-
\sum^{k_{0}}_{r=1}p^{-r/2}R_{r, p}\Big)s\Big\|_{L^{2}}
\leqslant C p^{-1}\big\|s\big\|_{L^{2}}.
\end{align}
A similar proof as for (\ref{83}) delivers
for $s\in L^{2}(X, L^{p}\otimes E)$,
\begin{align}\label{84}
\big\|R_{r, p}s\big\|_{L^{2}}\leqslant C\big\|s\big\|_{L^{2}},
\end{align}
which implies
\begin{align}\label{85}
\big\|\sum^{k_{0}}_{r=1}p^{-r/2}R_{r, p}s\big\|_{L^{2}}
\leqslant C p^{-1/2}\big\|s\big\|_{L^{2}}\ \
\textup{for}\ s\in L^{2}(X, L^{p}\otimes E),
\end{align}
for some constant $C>0$. Relations (\ref{83}) and (\ref{85}) entail 
the fist inequality of (\ref{4.27}),
which is equivalent to the second of (\ref{4.27}), by taking the adjoint.
This completes the proof of Lemma \ref{t4.5}.
\end{proof}

Let us consider the Taylor development of $\tilde{F}^{(i)}$ 
in normal coordinates around $x$
with $y=\exp^{X}_{x}(Z')$:
\begin{align}\label{4.37}
\tilde{F}^{(i)}(x, y)=\sum_{|\alpha|\leqslant k}
\frac{\partial^{\alpha}\tilde{F}^{(i)}}{\partial Z'^{\alpha}}(x, 0)
\frac{(\sqrt{p}Z')^{\alpha}}{\alpha!}p^{-|\alpha|/2}+
O(|Z'|^{k+1}).
\end{align}
The next step of the proof of Proposition \ref{t4.3} is the following.

\begin{lemma} \label{t4.6}
For every $j>0$, we have
\begin{align}\label{4.38}
\frac{\partial^{\alpha}\tilde{F}^{(i)}}{\partial Z'^{\alpha}}(x, 0)=0\ \
\textup{for}\ i-|\alpha|\geqslant j>0.
\end{align}
\end{lemma}

\begin{proof}
The definition (\ref{4.23}) of $\cT_{p}$ shows that
\begin{align}\label{89}
\cT^{\ast}_{p}=T_{p}-\sum_{i\leqslant \textup{deg}F_{x}}p^{i/2}
(P_{\mH_{p}}\tilde{F}^{(i)}).
\end{align}
Let us develop the sum on the right-hand side.
Consider the Taylor development (\ref{4.37})
with the expansion (\ref{0.28}) of the Bergman kernel we obtain
\begin{align}
& p^{-n}\big(P_{\mH_{p}}\tilde{F}^{(i)}\big)_{x_{0}}(0, Z')
\kappa^{1/2}(Z')
-\sum_{r, |\alpha|\leqslant k}\big(J_{r, x_{0}}\cP_{x_{0}}\big)
(0, \sqrt{p}Z')
\frac{\partial^{\alpha}\tilde{F}^{(i)}}{\partial Z'^{\alpha}}(x_{0}, 0)
\frac{(\sqrt{p}Z')^{\alpha}}{\alpha!}p^{-\frac{|\alpha|+r}{2}}
\\  &=
\Big[p^{-n}P_{\mH_{p}, x_{0}}(0, Z')\kappa^{1/2}(Z')
-\sum_{r\leqslant k}\big(J_{r, x_{0}}\cP_{x_{0}}\big)(0, \sqrt{p}Z')
p^{-r/2}\Big]
\tilde{F}_{x_{0}}^{(i)}(0, Z')
\nonumber
\\  & +
\sum_{r\leqslant k}\big(J_{r, x_{0}}\cP_{x_{0}}\big)(0, \sqrt{p}Z')
p^{-r/2} \Big[\tilde{F}_{x_{0}}^{(i)}(0, Z')
-
\sum_{|\alpha|\leqslant k}
\frac{\partial^{\alpha}\tilde{F}^{(i)}}{\partial Z'^{\alpha}}(x_{0}, 0)
\frac{(\sqrt{p}Z')^{\alpha}}{\alpha!}\Big].
\nonumber
\end{align}
By (\ref{0.28}), (\ref{42a}), (\ref{4.37}) and $\deg J_{r, x_{0}}
\leqslant 3r$,
we obtain for $k\geqslant \textup{deg}F_{x}+1$ and $m\in \N$, 
there exist
$\theta_{k, m}\in (0, 1)$ such that for any $Z'\in T_{x_{0}}X$
with $|Z'|\leqslant b p^{-\frac{1}{2}+\theta_{k, m}}$,
we have
\begin{align}\label{90}
& p^{-n}\sum_{i}\big(P_{\mH_{p}}\tilde{F}^{(i)}\big)_{x_{0}}(0, Z')
p^{i/2}
\\ &\cong 
\sum_{i}\sum_{|\alpha|, r\leqslant k}\big(J_{r, x_{0}}\cP_{x_{0}}\big)
(0, \sqrt{p}Z')
\frac{\partial^{\alpha}\tilde{F}^{(i)}}{\partial Z'^{\alpha}}(x_{0}, 0)
\frac{(\sqrt{p}Z')^{\alpha}}{\alpha!}p^{(i-|\alpha|-r)/2}
+\cO_{m}(p^{(\deg F-k)/2}).
\nonumber
\end{align}
Having in mind the second inequality of (\ref{4.27}),
this is only possible if for every $j>0$ the coefficients of $p^{j/2}$ on the
right-hand side of (\ref{90}) vanish. Thus, we have for every $j>0$:
\begin{align}\label{91}
\sum^{\deg F_{x}}_{l=j}
\sum_{|\alpha|+r=l-j}J_{r, x_{0}}(0, \sqrt{p}Z')
\frac{\partial^{\alpha}\tilde{F}^{(l)}}{\partial Z'^{\alpha}}(x_{0}, 0)
\frac{(\sqrt{p}Z')^{\alpha}}{\alpha!}=0.
\end{align}
From (\ref{91}), we will prove by recurrence that for any $j>0$,
(\ref{4.38}) holds.
As the first step of the recurrence let us take $j=\deg F_{x}$ in
(\ref{91}). Since $J_{0, x_{0}}=\Id_{E_{x_{0}}}$, then we get
immediately $\tilde{F}^{(\deg F_{x})}(x_{0}, 0)=0$.
Hence \eqref{4.38} holds for $j= \deg F_{x}$. 
Assume that $(\ref{4.38})$ holds for
$j>j_{0}>0$. Then for $j=j_{0}$, the coefficient with $r>0$ in (\ref{91})
is zero.
Since $J_{0, x_{0}}=\Id_{E_{x_{0}}}$, then (\ref{91}) reads
\begin{align}
\sum_{\alpha}\frac{\partial^{\alpha}
\tilde{F}^{(j_{0}+|\alpha|)}}{\partial Z'^{\alpha}}(x_{0}, 0)
\frac{(\sqrt{p}Z')^{\alpha}}{\alpha!}=0,
\end{align}
which entails (\ref{4.38}) for $j=j_{0}$. The proof of (\ref{4.38})
is complete.
\end{proof}

\begin{lemma}\label{t4.7}
For $i>0$, we have
\begin{align}\label{4.43}
\frac{\partial^{\alpha}F^{(i)}_{x}}{\partial \ov{z}'^{\alpha}}(0, 0)=0, 
\ |\alpha|\leqslant i.
\end{align}
Therefore, $F^{(i)}_{x}(0, \ov{z}')=0$ for all $i>0$ and $z'\in \C$, i.e.,
\textup{(\ref{4.18})} holds true. Moreover,
\begin{align}\label{4.44}
F^{(i)}_{x}(z, 0)=0 \ \textup{for all}\ i>0\ \textup{and all}\ z\in \C.
\end{align}
\end{lemma}

\begin{proof}
Let us start with some preliminary observations.
In view of (\ref{4.27}), (\ref{4.38}) and (\ref{90}), a comparison the
coefficient of $p^{0}$ in (\ref{4.11}) and (\ref{89}) yields
\begin{align}\label{95}
\tilde{F}^{(i)}(x, Z')=F^{(i)}_{x}(0, \ov{z}')+O(|Z'|^{i+1}).
\end{align}
Using the definition (\ref{4.20}) of $\tilde{F}^{(i)}(x, Z')$ and taking
the adjoint of (\ref{95}), we get
\begin{align}
F^{(i)}(Z', x)=\big(F^{(i)}_{x}(0, \ov{z}')\big)^{\ast}+O(|Z'|^{i+1}),
\end{align}
which implies
\begin{align}\label{4.47}
\frac{\partial^{\alpha}}{\partial z^{\alpha}}F^{(i)}(\cdot, x)\Big|_{x}
=\Big(\frac{\partial^{\alpha}}{\partial \ov{z}'^{\alpha}}
F^{(i)}_{x}(0, \ov{z}')\Big)^{\ast}\ \textup{for}\ |\alpha|\leqslant i.
\end{align}
In order to prove the Lemma it suffice to show that
\begin{align}\label{98}
\frac{\partial^{\alpha}}{\partial z^{\alpha}}F^{(i)}(\cdot, z)\Big|_{x}=0\
\textup{for}\ |\alpha|\leqslant i.
\end{align}
We prove this by induction over $|\alpha|$. For $|\alpha|=0$, it is obvious
that $F^{(i)}(0, x)=0$, since $F^{(i)}(0, x)$ is a homogeneous polynomial
of degree $i>0$. For the induction step,
let $j_{X}: X\rightarrow X\times X$ be the diagonal
injection. By Lemma \ref{t4.4} and the definition (\ref{4.20}) 
of $F^{(i)}(x, y)$,
\begin{align}\label{4.48}
\frac{\partial}{\partial z'_{j}}F^{i}(x, y)=0\ \textup{near}\ j_{X}(X),
\end{align}
where $y=\exp^{X}_{x}(Z')$. Assume now that 
$\alpha\in \N^{n}$ and
(\ref{98}) holds for $|\alpha|-1$. Consider $j$ with $\alpha_{j}>0$ and set
\begin{align}
\alpha'=(\alpha_{1}, \ldots, \alpha_{j}-1, \ldots, \alpha_{n}).
\end{align}
Taking the derivative of (\ref{4.20}) and using the induction hypothesis 
and (\ref{4.48}), we have
\begin{align}
\frac{\partial^{\alpha}}{\partial z^{\alpha}}F^{(i)}(\cdot, x)\Big|_{x}
=\frac{\partial}{\partial z_{j}}j^{\ast}_{X}
\big(\frac{\partial^{\alpha'}}{\partial z^{\alpha'}}F^{(i)}\big)\Big|_{x}
-\frac{\partial^{\alpha'}}{\partial z^{\alpha'}}
\frac{\partial}{\partial z'_{j}}F^{(i)}(\cdot, \cdot)\Big|_{0, 0}=0.
\end{align}
Thus, (\ref{4.43}) is proved. The identity (\ref{4.18}) follows too,
since it is equivalent to (\ref{4.43}).
Furthermore, (\ref{4.44}) follows from (\ref{4.18}) and (\ref{4.19}).
This finishes the proof of Lemma \ref{t4.7}.
\end{proof}

\begin{lemma}\label{t4.8}
We have $F^{(i)}_{x}(z, \ov{z}')=0$ for all $i>0$ and $z, z'\in \C^{n}$.
\end{lemma}

\begin{proof}
Let us consider the operator
\begin{align}\label{4.52}
\frac{1}{\sqrt{p}}P_{\mH_{p}}\big(\nabla_{X, x}^{E_{p}}T_{p}\big)
P_{\mH_{p}}
\ \ \textup{with}\ X\in \cC^{\infty}(X, TX),
X(x_{0})=\frac{\partial}{\partial z_{j}}+\frac{\partial}{\partial \ov{z}_{j}}.
\end{align}
By Remark \ref{t5},
the leading term of its asymptotic expansion as in (\ref{0.26}) is
\begin{align}
\Big(\frac{\partial}{\partial z_{j}}F_{x_{0}}\Big)(\sqrt{p}z, \sqrt{p}\ov{z}')
\cP_{x_{0}}(\sqrt{p}Z, \sqrt{p}Z').
\end{align}
By (\ref{4.18}) and (\ref{4.44}), 
$(\frac{\partial}{\partial z_{j}}F_{x_{0}})(z, \ov{z}')$ is an odd
polynomial in $z, \ov{z}'$ whose constant term vanishes. We
reiterate the argument from (\ref{4.23})--(\ref{4.47}) 
by replacing the operator $T_{p}$ with the operator
(\ref{4.52}); we get for $i>0$,
\begin{align}\label{4.54}
\frac{\partial}{\partial z_{j}}F^{(i)}_{x}(0, \ov{z}')=0.
\end{align}
By (\ref{4.19}) and (\ref{4.54}),
\begin{align}\label{4.55}
\frac{\partial}{\partial \ov{z}'_{j}}F^{(i)}_{x}(z, 0)=0.
\end{align}
By continuing this process, we show that for all $i>0$, 
$\alpha\in \Z^{n}$, $z, z'\in \C^{n}$,
\begin{align}\label{4.56}
\frac{\partial^{\alpha}}{\partial z^{\alpha}}F^{(i)}_{x}(0, \ov{z}')
=\frac{\partial^{\alpha}}{\partial \ov{z}'^{\alpha}}F^{(i)}_{x}(z, 0)=0.
\end{align}
Thus, Lemma \ref{t4.8} is proved and (\ref{4.17}) holds true.
\end{proof}

\noindent
Lemma \ref{t4.8} finishes the proof of Proposition \ref{t4.3}.
\end{proof}

We come now to the proof of the first induction step leading to (\ref{4.5}).

\begin{prop}\label{t4.9}
We have
\begin{align}\label{4.57}
p^{-n}(T_{p}-T_{g_{0}, p})_{x_{0}}(Z, Z')
\cong
\cO_{m}(p^{-1})
\end{align}
in the sense of Notation A. Consequently,
\begin{align}\label{4.58}
T_{p}=P_{\mH_{p}}g_{0}P_{\mH_{p}}+\mO(p^{-1}),
\end{align}
i.e., relation \textup{(\ref{4.8})} holds true in the
sense of \textup{(\ref{0.19})}.
\end{prop}

\begin{proof}
Let us compare the asymptotic expansion of $T_{p}$ and
$T_{g_{0}, p}=P_{\mH_{p}}g_{0}P_{\mH_{p}}$.
Using the Notation A, the expansion (\ref{0.29}) (for $k_{0}=2$) reads
for $\theta_{m}=1/8(n+k+m+2)$,
\begin{align}\label{4.59}
& p^{-n}T_{g_{0}, p, x_{0}}(Z, Z')
\\ \cong &
\big(g_{0}(x_{0})\cP_{x_{0}}+Q_{1, x_{0}}(g_{0})\cP_{x_{0}}p^{-1/2}
+Q_{2, x_{0}}(g_{0})\cP_{x_{0}}p^{-1}\big)
(\sqrt{p}Z, \sqrt{p}Z')+\cO_{m}(p^{-1}),
\nonumber
\end{align}
since $Q_{0, x_{0}}(g_{0})=g_{0}(x_{0})\Id_{E_{x_{0}}}$
by (\ref{3.17}).
The expansion (\ref{4.3}) (also for $k_{0}=2$) takes the form
for $\theta_{m}$ in (\ref{4.3}),
\begin{align}\label{4.60}
& p^{-n}T_{p, x_{0}}(Z, Z')
\\ \cong &
\big(g_{0}(x_{0})\cP_{x_{0}}+\cQ_{1, x_{0}}\cP_{x_{0}}p^{-1/2}
+\cQ_{2, x_{0}}\cP_{x_{0}}p^{-1}\big)
(\sqrt{p}Z, \sqrt{p}Z')+\cO_{m}(p^{-1}),
\nonumber
\end{align}
where we have used Proposition \ref{t4.3} and the
definition (\ref{4.6}) of $g_{0}$. Thus
subtracting (\ref{4.59}) from (\ref{4.60}) we obtain
for some $\theta_{m}\in (0, 1/2)$,
\begin{align}\label{4.61}
& p^{-n}(T_{p}-T_{g_{0}, p})_{x_{0}}(Z, Z')
\nonumber \\ \cong &
\big((\cQ_{1, x_{0}}-Q_{1, x_{0}}(g_{0}))
\cP_{x_{0}}\big)(\sqrt{p}Z, \sqrt{p}Z')p^{-1/2}
+\cO_{m}(p^{-1}).
\end{align}
Thus, it suffices to prove the following result.

\begin{lemma}\label{t4.10}
\begin{align}\label{4.63}
F_{1, x}:=\cQ_{1, x}-Q_{1, x}(g_{0})\equiv 0.
\end{align}
\end{lemma}

\begin{proof} We note first that $F_{1, x}$ is an odd polynomial in $z$
	and $\ov{z}'$;
we verify this statement as in Lemma \ref{t4.4}.
Thus the constant term of $F_{1, x}$ vanishes.
To show that the rest of the term vanish, we consider the decomposition
$F_{1, x}=\sum_{i\geqslant 0}F^{(i)}_{1, x}$ in homogeneous polynomials
$F^{(i)}_{1, x}$
of degree $i$. To prove (\ref{4.63}) it suffice to show that
\begin{align}\label{4.64}
F^{(i)}_{1, x}(z, \ov{z}')=0\ \ \textup{for all}\ i>0\ 
\textup{and}\ z, z'\in \C^{n}.
\end{align}
The proof of (\ref{4.64}) is similar to that of (\ref{4.17}).
Namely, we define as in (\ref{4.20}) the operator $F^{(i)}_{1}$,
by replacing $F^{(i)}_{x}(0, \ov{z}')$ by $F^{(i)}_{1, x}(0, \ov{z}')$, 
and we set (analogue to (\ref{4.23}))
\begin{align}
\cT_{p, 1}=T_{p}-P_{\mH_{p}}g_{0}P_{\mH_{p}}
-\sum_{i\leqslant \textup{deg}F_{i}}(F_{1}^{(i)}P_{\mH_{p}})p^{(i-1)/2}.
\end{align}
Due to (\ref{0.29}) and (\ref{4.3}), there exist polynomials 
$\tilde{R}_{r, x_{0}}\in \C[Z,Z']$ of the same
parity as $r$ such that the following expansion in the normal 
coordinates around $x_{0}\in X$
holds for any $k_{0}\geqslant 2$:
\begin{align}
p^{-n}\cT_{p, 1, x_{0}}(0, Z')
\cong
\sum^{k_{0}}_{r=2}(\tilde{R}_{r, x_{0}}\cP_{x_{0}})(0, \sqrt{p}Z')
p^{-r/2}
+\cO_{m}(p^{-k_{0}/2}),
\end{align}
in the sense of Condition A with $\theta_{k_{0}, m}$ the minimum of
$\theta_{k_{0}, m}$ in (\ref{0.29})
and $\theta_{k_{0}, m}$ in (\ref{4.3}).
This is an analogue of (\ref{4.25}). Now we can repeat with obvious
modifications the proof of (\ref{4.17})
and obtain the analogue of (\ref{4.17}) with $F_{x}$ 
replaced by $F_{1, x}$. This completes
the proof of Lemma \ref{t4.10}.
\end{proof}

Lemma \ref{t4.10} and the expansion (\ref{4.61}) imply 
immediately Proposition \ref{t4.9}.
\end{proof}

\begin{proof}[Proof of Theorem \textup{\ref{t4.1}}]
Proposition \ref{t4.9} shows that the asymptotic expansion (\ref{4.5}) 
of $T_{p}$ holds for $q=0$.
Moreover, if $T_{p}$ is self-adjoint, then from (\ref{4.59}), (\ref{4.60}), 
$g_{0}$ is also self-adjoint.
We show inductively that (\ref{4.5}) holds for every $q\in \N$.
To prove (\ref{4.5}) for $q=1$ let
us consider the operator $p(T_{p}-P_{\mH_{p}}g_{0}P_{\mH_{p}})$.
We have to show now that $p(T_{p}-P_{\mH_{p}}g_{0}P_{\mH_{p}})$
satisfies the hypotheses of Theorem \ref{t4.1}.
Due to Lemma \ref{t3.2} and Theorem \ref{t4.1} (ii),
the first two conditions are easily verified. To prove the third,
just subtract the
asymptotics of $T_{p, x_{0}}(Z, Z')$ (given by (\ref{4.3})) and 
$T_{g_{0}, p, x_{0}}(Z, Z')$
(given by (\ref{0.29})). Taking into account Proposition \ref{t4.3}
and (\ref{4.63}) the coefficients of $p^{0}$
and $p^{-1/2}$ in the difference vanish, which yields 
the desired conclusion.

Propositions \ref{t4.3} and \ref{t4.9} applied to 
$p(T_{p}-P_{\mH_{p}}g_{0}P_{\mH_{p}})$ yield
$g_{1}\in \cC^{\infty}(X, \End(E))$ such that (\ref{4.5}) holds
true for $q=1$.

We continue in this way the induction process to get (\ref{4.5}) for 
any $q$.
This completes the proof of Theorem \ref{t4.1}.
\end{proof}

\section{Algebra of Toeplitz operators}\label{s4}

The Poisson bracket $\{\cdot, \cdot\}$ on $(X, 2\pi \omega)$
is defined as
follows. For $f, g\in \cC^{\infty}(X)$, let $\xi_{f}$ be the Hamiltonian
vector field
generated by $f$, which is defined by $2\pi i_{\xi}\omega=df$. Then
\begin{align}\label{5.1}
\{f, g\}=\xi_{f}(dg).
\end{align}

\begin{proof}[Proof of Theorem \ref{t0}] First, it is obvious that
$P_{\mH_{p}}T_{f, p}T_{g, p}P_{\mH_{p}}=T_{f, p}T_{g, p}$.
To prove (\ref{4.1}), note that
from Lemma \ref{t3.2} and (\ref{0.29}), we know that for
any $k\in \N$ there exist $C_{k}>0$ and $M_{k}>0$ such that
for all $(x, x')\in X\times X$,
\begin{align}\label{5.1a}
\Big|T_{f, p}(x, x')\Big|_{\cC^{k}(X\times X)}\leqslant C_{k}p^{M_{k}}.
\end{align}
For any $b>0$ and $0<\theta<1$, if $d(x, x')>b p^{-\theta/2}$, then
\begin{align}\label{5.1b}
T_{f, p}T_{g, p}(x, x')=\Big(\int_{B^{X}(x, \frac{b}{2}p^{-\theta/2})}
+\int_{X\backslash B^{X}(x, \frac{b}{2}p^{-\theta/2})}\Big)
T_{f, p}(x, x'')T_{g, p}(x'', x')dv_{X}(x'').
\end{align}
Then (\ref{4.1}) follows from (\ref{0.23}),
(\ref{0.24b}), (\ref{5.1a}) and (\ref{5.1b}).
Like (\ref{3.23}),
for $|Z|, |Z'|<\frac{b}{2} p^{-\theta/2}$, we have
\begin{align}\label{5.2}
& (T_{f, p}T_{g, p})_{x_{0}}(Z, Z')
\\ =&
\int_{|Z''|<\var p^{-\theta_{1}/2}}T_{f, p, x_{0}}(z, z'')
T_{g, p, x_{0}}(Z'', Z')\kappa_{x_{0}}(Z'')dv_{TX}(Z'')
+\cO_{m}(p^{-\infty}).
\nonumber
\end{align}
By Lemma \ref{t3.2} and Lemma \ref{t3.4} and (\ref{5.2}),
we deduce as we obtain (\ref{3.46}) from
Proposition \ref{t1}, (\ref{0.18a}) and (\ref{3.23})
in the proof of Lemma \ref{t3.4}
that for $|Z|, |Z'|<\frac{b}{2} p^{-\theta/2}$, we have
\begin{align}\begin{split}\label{5.3}
 & p^{-n}(T_{f, p}T_{g, p})_{x_{0}}(Z, Z')
\cong 
\sum^{k_{0}}_{r=0}\big(Q_{r, x_{0}}(f, g)\cP_{x_{0}}\big)
(\sqrt{p}Z, \sqrt{p}Z')p^{-r/2}
+\cO_{m}(p^{-k_{0}/2}),
\end{split}\end{align}
with
\begin{align}\label{Q}
Q_{r, x_{0}}(f, g)=\sum_{r_{1}+r_{2}=r}\cK[Q_{r_{1}, x_{0}}(f), 
Q_{r_{2}, x_{0}}(g)].
\end{align}
Thus, $T_{f, p}T_{g, p}$ is a Toeplitz operator by Theorem \ref{t4.1}.
Moreover, it follows form the proofs of Lemma \ref{t3.4} 
and Theorem \ref{t4.1}
that $g_{l}=C_{l}(f, g)$, where $C_{l}$ are bidifferential operators.

The rest of the proof of Theorem \ref{t0} is exactly the same
as that of \cite[Theorem 1.1]{MM08} and we omit it here.
This finishes the proof of Theorem \ref{t0}.
\end{proof}

\begin{proof}[Proof of Theorem \ref{t2}]
Take a point $x_0\in X$ and $u_0\in E_{x_0}$ with $|u_0|_{h^{E}}=1$ 
such that $|f(x_0)(u_0)|={\norm f}_\infty$.
Recall that we trivialized the bundles $L$, $E$ 
in normal coordinates near $x_0$, 
and $e_L$ is the unit frame of $L$ which trivializes $L$. 
Moreover, in this normal coordinates, $u_0$ is a trivial section of $E$.
Considering the sequence of sections 
$S^p_{x_0}=p^{-n/2}P_{\mH_p}(e_L^{\otimes p}\otimes u_0)$, 
we have by \eqref{0.28},
\begin{equation}\label{toe4.18}
\big\|T_{f,\,p}\,S^p_{x_0}-f(x_0)S^p_{x_0}\big\|_{L^2}
\leqslant \frac{C}{\sqrt{p}}\norm{S^p_{x_0}}_{L^2}\,, 
\end{equation}
which immediately implies \eqref{toe4.17}.
\end{proof}

\section{Proof of Theorem \ref{t3}}\label{s6}
In this section, we show how to adapt the results of \cite{Ioos}
in order to prove Theorem \ref{t3}, which provides formula \eqref{c1fg} for
the coefficient $C_1(f,g)$ in the Toeplitz expansion \eqref{toe4.2} of the product
$T_{f,p}T_{g,p}$ of two Toeplitz operators (cf.\ Theorem \ref{t0}),
 when $g^{TX}(\cdot,\cdot)=\omega(\cdot, J\cdot)$. 

Let us introduce the main tools needed for this computation. 
It is shown in \cite[Theorem 1.4]{Ma08} that the restriction on 
$B^X(x,\epsilon)$ of the operator $\Delta_{p, \Phi}$ 
defined in \eqref{0.4} is equal, through the trivializations given
in Section \ref{s1} and after a convenient rescaling in $\sqrt{p}:=1/t$, 
to an operator $\cL_t$ on $B^{T_xX}(0,\epsilon)$ satisfying
\begin{equation}\label{expLt2}
\cL_t=\cL_0+\sum_{r=1}^mt^r \mO_r +\cO(t^{m+1}),
\end{equation}
for any $m\in\N$, where $\{\mO_r\}_{r\in\N}$ is a family of 
differential operators of order equal or less than $2$, 
with coefficients explicitly computable in terms of local data, 
and where the differential operator $\cO(t^{m+1})$ has
its coefficients and their derivatives up to order $k$ dominated 
by $C_k t^{m+1}$ for any $k\in\N$.

Moreover, the differential operator $\cL_0$ acts on the scalar 
part of smooth functions on $\R^{2n}$ with values in $E_x$, 
has a positive discrete spectrum on $L^2(\R^{2n})$, 
and the kernel of the orthogonal projection $\cP$ on $\Ker(\cL_0)$
is given by \eqref{p} with $a_j=2\pi$ therein.
 We write $\cP^\perp=\Id-\cP$, and  
then define an operator $(\cL_0)^{-1}\cP^\perp$ by inverting
the eigenvalues of $\cL_0$ on $\Ker(\cL_0)^\perp$.

As shown in \cite{Ma08}, there is a direct method to compute 
the family $\{\cF_{r,x}(Z,Z')\}_{r\in\N}$ defined in Theorem \ref{t22}, 
using \eqref{expLt2}. The following lemma, which has been established 
in \cite[Theorem 1.4, Theorem 1.16, (1.111)]{Ma08}, gives the first
three elements of this family. 

\begin{lemma}\label{J}
For all $r\in\N$, let $\cF_{r,x}$ be the operator associated 
with the kernel $\cF_{r,x}(Z,Z')$, and let the differential operators 
$\mO_1$ and $\mO_2$ be as in \eqref{expLt2}. 
Then the following formulas hold:
\begin{equation}
\label{J0}
\cF_{0,x}=\cP,
\end{equation}

\begin{equation}
\label{J1}
\cF_{1,x}=-(\cL_0)^{-1}\cP^{\perp}\mO_1\cP-\cP\mO_1 (\cL_0)^{-1}
\cP^{\perp},
\end{equation}

\begin{equation}
\label{J2}
\begin{split}
\cF_{2,x}&=(\cL_0)^{-1}\cP^{\perp}\mO_1 (\cL_0)^{-1}\cP^{\perp}
\mO_1 \cP-(\cL_0)^{-1}\cP^{\perp}\mO_2 \cP\\
& +\cP\mO_1 (\cL_0)^{-1}\cP^{\perp}\mO_1 (\cL_0)^{-1}\cP^{\perp}
-\cP\mO_2 (\cL_0)^{-1}\cP^{\perp}\\
+ & (\cL_0 )^{-1}\cP^{\perp}\mO_1 \cP\mO_1 (\cL_0)^{-1}\cP^{\perp}
-\cP\mO_1(\cL_0)^{-2}\cP^{\perp}\mO_1 \cP.
\end{split}
\end{equation}

Moreover, $\mO_1$ commutes with any $A\in\End(E_x)$, and we have 
the formula
\begin{equation}
\label{POP=0}
\cP\mO_1 \cP=0.
\end{equation}

In particular, $\cF_{0,x}$ and $\cF_{1,x}$ commute with
any $A\in\End(E_x)$.

\end{lemma}

Lemma \ref{J} corresponds to \cite[Lemma 3.1]{Ioos}, and
the following technical Lemma corresponds to \cite[Lemma 3.3]{Ioos}.
It was essentially proved in \cite[(2.25)]{Ma08}.

\begin{lemma}
\label{L} The following formulas hold:
\begin{equation}
\label{L1}
\begin{split}
(\cP\mO_1 (\cL_0)^{-1}\cP^{\perp})(0,Z')=0,\\
(\cP\mO_1 (\cL_0)^{-1}\cP^{\perp})(Z,0)=0,\\
((\cL_0)^{-1}\cP^{\perp}\mO_1\cP)(Z,0)=0.
\end{split}
\end{equation}
\end{lemma}

The result of Lemma \ref{L} is a simple computation using 
\cite[(1,98),(1.99),(2.25)]{Ma08}. In fact, all the kernels associated to
the situation in this paper are 
the scalar part of the kernels of the corresponding situation in 
\cite{Ioos}. Lemma \ref{L} is then an expression of the fact that the 
corresponding formulas in \cite[Lemma 3.3]{Ioos} have no scalar parts.

Now note that by the proof of Theorem \ref{t4.1} at the end of
Section \ref{s4}, we still have the following formula:
\begin{equation}\label{Q2(f,g)-Q2(fg)}
C_1(f,g)(x)=Q_2(f,g)(0,0)-Q_2(fg)(0,0),
\end{equation}
where the coefficients $Q_2(fg)$ and $Q_2(f,g)$ has been defined in
Lemma 
\ref{t3.4} and \eqref{Q} respectively. Note that the formula 
\eqref{Q2(f,g)-Q2(fg)} is actually simpler than the one 
given in \cite{MM08}, 
as we only need to consider its scalar part.

Note that the reasoning of Section \ref{s4} to deduce \eqref{Q}
from the proof of Lemma \ref{t3.4} says that the composition of 
kernels is 
still computable in the same way than in \cite[\S\ 4.3]{MM08}, 
so that the kernel calculus as described in \cite[\S\ 2.3]{Ioos} 
still works in this context.

Using Lemma \ref{J}, Lemma \ref{L} and the Sections \ref{s3} 
and \ref{s4}, it is then straightforward to adapt the arguments 
of \cite[\S\ 2.2]{Ioos}. Thanks to Lemma \ref{L}, 
the proofs of \cite[Proposition 3.2, Proposition 3.3]{Ioos} are even easier. 
We thus get \cite[Theorem 1.1]{Ioos} in the context described
in the introduction, which is exactly Theorem \ref{t3}.

\end{document}